\documentclass[oneside, a4paper, 10pt]{article}
\usepackage[colorlinks = true,
            linkcolor = black,
            urlcolor  = blue,
            citecolor = black]{hyperref}
\usepackage{amsmath, amsthm, amssymb, amsfonts, amscd}
\usepackage[left=2cm,right=2cm,top=3cm,bottom=3cm,a4paper]{geometry}
\usepackage{mathtools}
\usepackage{thmtools}
\usepackage{graphicx}
\usepackage{cleveref}
\usepackage{mathrsfs}
\usepackage{titling}
\usepackage{url}
\usepackage[nosort, noadjust]{cite}
\usepackage{enumitem}
\usepackage{tikz-cd}
\usepackage{multirow}
\usepackage{fancyhdr}
\usepackage{sectsty}
\usepackage{longtable}
\usepackage[title]{appendix}
\newcommand{\im}{{\operatorname{im}}}
\newcommand{\cok}{{\operatorname{cok}}}

\newcommand{\End}{\operatorname{End}}
\newcommand{\Aut}{\operatorname{Aut}}
\newcommand{\GL}{\operatorname{GL}}

\newcommand{\rank}{\operatorname{rank}}

\newcommand{\diag}{\operatorname{diag}}
\newcommand{\Sur}{\operatorname{Sur}}

\newcommand{\Prs}{\operatorname{Prs}}
\newcommand{\M}{\operatorname{M}}
\newcommand{\Z}{\mathbb{Z}}
\newcommand{\extz}{\overline{\Z}}

\newcommand{\C}{\mathbb{C}}
\newcommand{\F}{\mathbb{F}}
\newcommand{\PP}{\mathbb{P}}
\newcommand{\EE}{\mathbb{E}}
\newcommand{\GG}{\mathcal{G}}
\newcommand{\calA}{\mathcal{A}}
\newcommand{\calB}{\mathcal{B}}
\newcommand{\calC}{\mathcal{C}}
\newcommand{\calM}{\mathcal{M}}
\newcommand{\calP}{\mathcal{P}}
\newcommand\undermat[2]{
  \makebox[0pt][l]{$\smash{\underbrace{\phantom{%
    \begin{matrix}#2\end{matrix}}}_{\text{$#1$}}}$}#2}

\newcommand{\tred}{\textcolor{red}}

\theoremstyle{definition}
\newtheorem{theorem}{Theorem}[section]
\newtheorem{proposition}[theorem]{Proposition}
\newtheorem{problem}[theorem]{Problem}
\newtheorem{corollary}[theorem]{Corollary}
\newtheorem{lemma}[theorem]{Lemma}
\newtheorem{conjecture}[theorem]{Conjecture}
\newtheorem*{conjecture*}{Conjecture}
\newtheorem{remark}[theorem]{Remark}
\newtheorem{definition}[theorem]{Definition}
\newtheorem{example}[theorem]{Example}

\makeatletter 
\newcommand\semilarge{\@setfontsize\semilarge{11}{13.2}}
\makeatother

\title{\semilarge{\textbf{JOINT DISTRIBUTION OF THE COKERNELS OF RANDOM $p$-ADIC MATRICES II}} \vspace{-2mm}}
\author{\normalsize{JIWAN JUNG, JUNGIN LEE} }
\date{}

\newcommand\shorttitle{JOINT DISTRIBUTION OF RANDOM $p$-ADIC MATRICES II}
\newcommand\authors{JIWAN JUNG, JUNGIN LEE}

\sectionfont{\centering \large}
\subsectionfont{\normalsize}
\makeatletter
\renewcommand{\@seccntformat}[1]{\csname the#1\endcsname.\quad}
\makeatother

\renewenvironment{abstract}
 {\quotation\small\noindent\rule{\linewidth}{.5pt}\par\smallskip
  {\centering\bfseries\abstractname\par}\medskip}
 {\par\noindent\rule{\linewidth}{.5pt}\endquotation}
 
\AtEndDocument{%
  \bigskip
  \footnotesize

  \textsc{Jiwan Jung, Department of Mathematics, Pohang University of Science and Technology, Pohang 37673, Republic of Korea}\par\nopagebreak
  \hspace{3mm} \textit{E-mail address:} \texttt{guinipig123@postech.ac.kr}

  \textsc{Jungin Lee, Department of Mathematics, Ajou University, Suwon 16499, Republic of Korea}\par\nopagebreak
  \hspace{3mm} \textit{E-mail address:} \texttt{jileemath@ajou.ac.kr}}
\begin{document}
\maketitle
\vspace{-18mm}

\begin{abstract}
In this paper, we study the combinatorial relations between the cokernels $\text{cok}(A_n+px_iI_n)$ ($1 \le i \le m$) where $A_n$ is an $n \times n$ matrix over the ring of $p$-adic integers $\mathbb{Z}_p$, $I_n$ is the $n \times n$ identity matrix and $x_1, \cdots, x_m$ are elements of $ \mathbb{Z}_p$ whose reductions modulo $p$ are distinct. For a positive integer $m \le 4$ and given $x_1, \cdots, x_m \in \mathbb{Z}_p$, we determine the set of $m$-tuples of finitely generated $\mathbb{Z}_p$-modules $(H_1, \cdots, H_m)$ for which $(\text{cok}(A_n+px_1I_n), \cdots, \text{cok}(A_n+px_mI_n)) = (H_1, \cdots, H_m)$ for some matrix $A_n$. We also prove that if $A_n$ is an $n \times n$ Haar random matrix over $\mathbb{Z}_p$ for each positive integer $n$, then the joint distribution of $\text{cok}(A_n+px_iI_n)$ ($1 \le i \le m$) converges as $n \rightarrow \infty$.

\end{abstract}

\section{Introduction} \label{Sec1}
\allowdisplaybreaks

Friedman and Washington \cite{FW89} computed the distribution of the cokernel of a random matrix over the ring of $p$-adic integers $\Z_p$. They proved that if $A_n \in \M_n(\Z_p)$ is a Haar random matrix (equidistributed with respect to the Haar measure) for each positive integer $n$ and $H$ is a finite abelian $p$-group, then
\begin{equation} \label{eq1a}
\lim_{n \rightarrow \infty} \PP (\cok (A_n) \cong H) = \frac{\prod_{k=1}^{\infty}(1-p^{-k})}{\left| \Aut(H) \right|}.
\end{equation}
Here $\M_{m \times n}(R)$ denotes the set of $m \times n$ matrices over a commutative ring $R$, $\M_n(R) := \M_{n \times n}(R)$ and $\PP( \, \cdot \, )$ denotes the probability of an event. The study of the distributions of the cokernels for much larger classes of random $p$-adic matrices was initiated by the work of Wood \cite{Woo17} which proved universality for random symmetric matrices over $\Z_p$. Precisely, Wood proved that if $A_n \in \M_n(\Z_p)$ is an $\varepsilon$-balanced random symmetric matrix for each positive integer $n$, then the distribution of $\cok(A_n)$ always converges to the same distribution as $n \rightarrow \infty$. 

\begin{definition} \label{def1a}
For a real number $0 < \varepsilon < 1$, a random variable $x \in \Z_p$ is $\varepsilon$\textit{-balanced} if $\PP (x \equiv r \,\, (\text{mod } p)) \le 1 - \varepsilon$ for every $r \in \Z / p\Z$. 
A random matrix $A \in \M_n(\Z_p)$ is $\varepsilon$\textit{-balanced} if its entries are independent and $\varepsilon$-balanced. A random symmetric matrix $A \in \M_n(\Z_p)$ is $\varepsilon$\textit{-balanced} if its upper triangular entries are independent and $\varepsilon$-balanced. 
\end{definition}

\begin{theorem} \label{thm1b}
(\cite[Corollary 9.2]{Woo17}) Let $0 < \varepsilon < 1$ be a real number, $H$ be a finite abelian $p$-group and $A_n \in \M_n(\Z_p)$ be an $\varepsilon$-balanced random symmetric matrix for each $n$. Then we have
    $$
    \lim_{n \rightarrow \infty} \PP(\cok(A_n) \cong H) = \frac{\# \left\{ \text{symmetric, bilinear, perfect } \phi : H \times H \rightarrow \C^* \right\}}{\left | H \right | \left | \Aut(H) \right |} \prod_{k=1}^{\infty}(1-p^{1-2k}).
    $$
\end{theorem}

One of the key ingredients of the proof of Theorem \ref{thm1b} is the use of moments for random finitely generated abelian groups. For a given finite abelian group $H$, the $H$-\textit{moment} of a random finitely generated abelian group $X$ is defined by the expected value $\EE(\# \Sur(X, H))$ of the number of surjective homomorphisms from $X$ to $H$. 
If the moments of a random finitely generated abelian group $X$ are not too large, then the distribution of $X$ is uniquely determined by its moments \cite[Theorem 8.3]{Woo17}. Theorem \ref{thm1b} follows from this result and a sophisticated computation of the moments of the cokernels of $\varepsilon$-balanced matrices.

Starting from the work of Wood, several universality results for the cokernels of random $p$-adic matrices were proved \cite{CY23, Lee23b, Lee24, NVP24, NW22a, NW22b, Woo19, Yan23}. 
All of these results were obtained by computing the (mixed) moments of the cokernels and determining the (joint) distribution of the cokernels from the moments.
As an example, we provide a theorem of Nguyen and Wood \cite{NW22a} which proves universality for $\varepsilon_n$-balanced matrices over $\Z_p$ where $\varepsilon_n$ does not decrease too fast as $n \rightarrow \infty$. 

\begin{theorem} \label{thm1c}
(\cite[Theorem 4.1]{NW22a}) Let $u \ge 0$ be an integer, $H$ be a finite abelian $p$-group and $(\varepsilon_n)_{n \ge 1}$ be a sequence of real numbers such that $0 < \varepsilon_n < 1$ for each $n$ and for every $\Delta > 0$, we have $\varepsilon_n \ge \frac{\Delta \log n}{n}$ for sufficiently large $n$. Let $A_n \in \M_{n \times (n+u)}(\Z_p)$ be an $\varepsilon_n$-balanced random matrix for each $n$. Then we have
\begin{equation} \label{eq1b}
\lim_{n \rightarrow \infty} \PP(\cok(A_n) \cong H) = \frac{\prod_{k=1}^{\infty} (1-p^{-k-u})}{\left| H \right|^{u} \left| \Aut(H) \right| }.
\end{equation}
\end{theorem}

On the other hand, there had been recent progress on generalization of the cokernel condition. Friedman and Washington \cite{FW89} proved that if $A_n$ is a Haar random matrix in $\GL_n(\Z_p)$ for each $n$ and $H$ is a finite abelian $p$-group, then
\begin{equation} \label{eq1c}
\lim_{n \rightarrow \infty} \PP (\cok(A_n - I_n) \cong H) 
= \frac{\prod_{k=1}^{\infty}(1-p^{-k})}{\left| \Aut(H) \right|}.    
\end{equation}
($I_n$ denotes the $n \times n$ identity matrix.) As a natural generalization of this result, Cheong and Huang \cite{CH21} predicted the limiting joint distribution of the cokernels $\cok(P_i(A_n))$ ($1 \le i \le m$) where $A_n \in \M_n(\Z_p)$ is a Haar random matrix for each $n$ and $P_1(t), \cdots, P_m(t) \in \Z_p[t]$ are monic polynomials whose reductions modulo $p$ are distinct and irreducible. This conjecture was settled by the second author \cite[Theorem 2.1]{Lee23a}. (Cheong and Kaplan \cite[Theorem 1.1]{CK22} independently proved the conjecture under the assumption that $\deg(P_i) \le 2$ for each $i$.) Recently, Cheong and Yu \cite{CY23} generalized this to the case that $A_n$ is $\varepsilon$-balanced for each $n$.

\begin{theorem} \label{thm1d}
(\cite[Corollary 1.8]{CY23}) Let $0 < \varepsilon < 1$ be a real number and $A_n \in \M_n(\Z_p)$ be an $\varepsilon$-balanced matrix for each $n \ge 1$. Let $P_1(t), \cdots, P_m(t) \in \Z_p[t]$ be monic polynomials whose reductions modulo $p$ in $\F_p[t]$ are distinct and irreducible. Also let $H_i$ be a finite module over $R_i := \Z_p[t]/(P_i (t))$ for each $i$. Then we have
\begin{equation} \label{eq1d}
\lim_{n \rightarrow \infty} \PP \begin{pmatrix}
\cok(P_i(A_n)) \cong H_i \\
\text{for } 1 \le i \le m
\end{pmatrix} = \prod_{i=1}^{m} \frac{\prod_{k=1}^{\infty} (1-p^{-k \deg(P_i)})}{\left| \Aut_{R_i}(H_i) \right|}.
\end{equation}
\end{theorem}

We remark that each $R_i$ is a discrete valuation ring with a finite residue field $R_i/pR_i \cong \F_{p^{\deg(P_i)}}$ and the cokernel $\cok(P_i(A_n))$ has a natural $R_i$-module structure defined by $t \cdot x := A_n x$. There are other ways to generalize the cokernel condition. For example, Van Peski \cite[Theorem 1.4]{VP23} computed the joint distribution of
$$
\cok(A_1), \, \cok(A_2A_1), \cdots, \, \cok(A_r \cdots A_1)
$$
for a fixed $n \ge 1$ and Haar random matrices $A_1, \cdots, A_r \in \M_n(\Z_p)$ by using explicit formulas for certain skew Hall-Littlewood polynomials. Nguyen and Van Peski \cite[Theorem 1.2]{NVP24} generalized this to the case where $A_1, \cdots, A_r$ are $\varepsilon$-balanced.

In Theorem \ref{thm1d}, the distribution of the cokernels $\cok(P_i(A_n))$ ($1 \le i \le m$) becomes asymptotically independent as $n \rightarrow \infty$. Here the condition that the reductions modulo $p$ of $P_1(t), \cdots, P_m(t)$ are distinct is essential. 
If two polynomials $P_1(t), P_2(t) \in \Z_p[t]$ have the same reduction modulo $p$, then $\cok(P_1(A))$ and $\cok(P_2(A))$ have the same $p$-rank so they cannot be asymptotically independent. (The $p$-rank of a finite abelian $p$-group $G$ is given by $r_p(G) := \rank_{\F_p}(G / pG)$.) Nevertheless, we can still consider their joint distribution. 
In the previous work of the second author \cite{Lee24}, the joint distribution in the simplest case ($P_1(t)=t$ and $P_2(t)=t+p$) was computed. 
Denote $c_r(p) := \prod_{k=1}^{r} (1-p^{-k})$ and $c_{\infty}(p) := \prod_{k=1}^{\infty} (1-p^{-k})$. 

\begin{theorem} \label{thm1e}
(\cite[Theorem 3.11]{Lee24}) Let $(\varepsilon_n)_{n \ge 1}$ be a sequence of real numbers such that for every $\Delta > 0$, we have $\varepsilon_n \ge \frac{\Delta \log n}{n}$ for sufficiently large $n$. Let $A_n \in \M_n(\Z_p)$ be an $\varepsilon_n$-balanced random matrix for each $n$. Then we have
\begin{equation*} 
\lim_{n \rightarrow \infty} \PP \begin{pmatrix}
\cok(A_n) \cong H_1 \text{ and} \\
\cok(A_n+pI_n) \cong H_2
\end{pmatrix} = \left\{\begin{matrix}
0 & (r_p(H_1) \neq r_p(H_2)) \\
\frac{p^{r^2} c_{\infty}(p)c_{r}(p)^2}{\left| \Aut(H_1) \right| \left| \Aut(H_2) \right|} & (r_p(H_1) = r_p(H_2) = r)
\end{matrix}\right.
\end{equation*}
for every finite abelian $p$-groups $H_1$ and $H_2$.
\end{theorem}

It is very hard to compute the joint distribution of the cokernels $\cok(P_i(A))$ ($1 \le i \le m$) in general, even in the case that each $A_n$ is equidistributed. Thus we propose the following easier problem. 

\begin{problem} \label{prob1f}
Let $P_1(t), \cdots, P_m(t) \in \Z_p[t]$ be monic polynomials whose reductions modulo $p$ are irreducible, $R_i=\Z_p[t]/(P_i(t))$, $\calM_{R_i}$ be the set of finitely generated $R_i$-modules and $\calM = \prod_{i=1}^{m} \calM_{R_i}$. For a given $(H_1, \cdots, H_m) \in \calM$, determine whether there exists a matrix $A_n \in \M_n(\Z_p)$ such that $\cok(P_i(A_n))\cong H_i$ for each $i$. In other words, determine the set 
$$
\calC(P_1, \cdots, P_m) := \begin{Bmatrix}
(H_1, \cdots, H_m) \in \calM : \text{there exists } A_n \in \M_n(\Z_p) \text{ for some } n\\
\text{such that } \cok(P_i(A_n)) \cong H_i \text{ for each } 1 \le i \le m
\end{Bmatrix}.
$$
\end{problem}

\begin{remark} \label{rmk1g}
\begin{enumerate}
    \item For $A \in \M_{n}(\Z_p)$ and $B \in \M_{n'}(\Z_p)$, we have $\cok \left ( P_i \begin{pmatrix}
A & O\\ 
O & B
\end{pmatrix} \right ) \cong \cok(P_i(A)) \times \cok(P_i(B))$ so the set $\calC(P_1, \cdots, P_m)$ is closed under componentwise finite direct product. 
    
    \item In the above problem, we allow the case that $\cok(P_i(A_n))$ have a free part (i.e. $\det(P_i(A_n)) = 0$), contrary to Theorem \ref{thm1d}. The probability that $\det(P_i(A_n)) = 0$ for some $i$ is always zero, but it does not mean that this event cannot happen. 
\end{enumerate}
\end{remark}

In this paper, we analyze the case where $P_i(t) = t + px_i$ for some $x_1, \cdots, x_m \in \Z_p$ whose reductions modulo $p$ are distinct. 
Let $X_m := \left \{ x_1, \cdots, x_m \right \}$ be a finite ordered subset of $\Z_p$ of size $m$ whose elements have distinct reductions modulo $p$ and denote $\calC_{X_m} := \calC(t+px_1, \cdots, t+px_m) \subset \calM_{\Z_p}^{m}$. The main result of the paper is the following theorem, which determines the set $\calC_{X_m}$ for $m \le 4$. Note that in each case, the set $\calC_{X_m}$ does not depend on the choice of $X_m$.

\begin{theorem} \label{thm1h}
For $(H_1, \cdots, H_m) \in \calM_{\Z_p}^m$, write $H_i \cong \Z_p^{d_{\infty, i}} \times \prod_{r=1}^{\infty} (\Z/p^r\Z)^{d_{r, i}}$, $D_r := \sum_{i=1}^{m} d_{r, i}$ and $s_i := \rank_{\F_p}(H_i/pH_i) = \sum_{r=1}^{\infty} d_{r,i} + d_{\infty, i}$ for each $i$.
\begin{enumerate}
    \item \label{item1g1} $\calC_{X_1} = \calM_{\Z_p}$.

    \item \label{item1g2} $\calC_{X_2} = \left \{ (H_1, H_2) \in \calM_{\Z_p}^2 : s_1=s_2 \right \}$.

    \item \label{item1g3} $\calC_{X_3} = \left \{ (H_1,H_2,H_3) \in \calM_{\Z_p}^3 : s_1=s_2=s_3 \text{ and } 2d_{1, i} \le D_1 \;\; (1 \le i \le 3) \right \}$.

    \item \label{item1g4} $\calC_{X_4} = \left \{ (H_1, H_2, H_3, H_4) \in \calM_{\Z_p}^4 : \begin{matrix}
s_1=s_2=s_3=s_4, \; 3d_{1, i} \le D_1 \;\; (1 \le i \le 4) \text{ and} \\
d_{1, i} + 2(d_{1,j}+d_{2,j}) \le D_1+D_2 \;\; (1 \le i,j \le 4)
\end{matrix} \right \}$.
\end{enumerate}
\end{theorem}

If $(H_1, \cdots, H_m) \in \calC_{X_m}$, then there exists $A_n \in \M_n(\Z_p)$ such that $\cok(A_n+px_iI_n) \cong H_i$ for each $i$. In this case, we have $(\Z/p\Z)^{s_i} \cong H_i/pH_i \cong \cok_{\F_p}(\overline{A_n})$ where $\overline{A_n} \in \M_n(\F_p)$ is the reduction modulo $p$ of $A_n$, which implies that $s_1=s_2=\cdots = s_m$. When $m=2$, this is the only condition we need for elements of $\calC_{X_m}$. We need some additional relations between the numbers $d_{r, i}$ for larger $m$. It is natural to suggest the following conjecture from Theorem \ref{thm1h}.

\begin{conjecture} \label{conj1i}
Let $d_{r,i}$, $D_r$ and $s_i$ be as in Theorem \ref{thm1h}. For every integer $m \ge 1$, we have $(H_1, \cdots, H_m) \in \calC_{X_m}$ if and only if $s_1=\cdots=s_m$ and 
$$
\sum_{k=1}^{r-1}\left ( \sum_{l=1}^{k} d_{l, i_k} \right ) + (m-r) \sum_{l=1}^{r} d_{l, i_r}
\le D_1+ \cdots + D_r
$$
for every $1 \le r \le m-2$ and $1 \le i_1, \cdots, i_r \le m$.
\end{conjecture}

The paper is organized as follows. In Section \ref{Sub21}, we provide basic notations and provide some basic properties of the sets $\calC_{X_m,l,k}$ (see Problem \ref{prob2b}) which are helpful to understand the set $\calC_{X_m}$. We prove the main result of the paper (Theorem \ref{thm1h}) for the case $m \le 3$ in Section \ref{Sub22}. 
The technical heart of the paper is a reduction procedure, which is explained in Section \ref{Sub31}. Using this reduction procedure and explicit linear-algebraic computations on matrices over $\Z_p$, we prove the necessary condition of Theorem \ref{thm1h} for $m=4$ in Section \ref{Sub32}. After that, we prove the sufficient condition of Theorem \ref{thm1h} for $m=4$ in Section \ref{Sub33} based on a zone theory.

The last section is devoted to the joint distribution of the cokernels $\cok(A_n+px_iI_n)$ ($1 \le i \le m$). In Section \ref{Sub41}, we prove that if each $A_n$ is a Haar random matrix, then the joint distribution of the cokernels $\cok(A_n+px_iI_n)$ ($1 \le i \le m$) converges as $n \rightarrow \infty$. 
In fact, we prove the following more general result. Note that our proof does not provide the limiting joint distribution.
\begin{theorem} \label{thm1j}
(Theorem \ref{thm5a}) Let $A_n \in \M_n(\Z_p)$ be a Haar random matrix for each $n \ge 1$, $y_1, \cdots, y_m$ be distinct elements of $\Z_p$ and $H_1, \cdots, H_m$ be finite abelian $p$-groups. Then the limit 
$$
\lim_{n \rightarrow \infty} \PP \begin{pmatrix}
\cok(A_n + y_i I_n) \cong H_i \\
\text{for } 1 \le i \le m
\end{pmatrix}
$$
converges.
\end{theorem}

In Section \ref{Sub42}, we compute the mixed moments of the cokernels $\cok(A_n+px_iI_n)$ ($1 \le i \le m$) where each $A_n \in \M_n(\Z_p)$ is given as in Theorem \ref{thm1e}. 
Then it is natural to follow the proof of Theorem \ref{thm1e} given in \cite{Lee24}, where the second author determined the unique joint distribution of $\cok(A_n)$ and $\cok(A_n+pI_n)$ from their mixed moments. However, it turns out that for $m \ge 3$, we cannot determine a unique joint distribution of $\cok(A_n+px_iI_n)$ ($1 \le i \le m$) from their mixed moments using existing methods (see Example \ref{ex6d}). 

Let $Y$ be a random $m$-tuple of finite abelian $p$-groups (or a random $m$-tuple of finitely generated $\Z_p$-modules in general). When $Y$ is supported on a smaller set of $m$-tuples of finite abelian $p$-groups, it is more likely that the distribution of $Y$ is uniquely determined by its mixed moments. Therefore the information on the support of $Y$ would be helpful for determining its distribution. This is one of our motivations for concerning Problem \ref{prob1f} in this paper. In the future work, we hope to determine the joint distribution of $\cok(A_n+px_iI_n)$ ($1 \le i \le m$) from their mixed moments, together with combinatorial relations between the cokernels provided in Theorem \ref{thm1h} and Conjecture \ref{conj1i}. 

\section{Preliminaries} \label{Sec2}


The following notations will be used throughout the paper.
\begin{itemize}
    \item For a prime $p$, let $\GG_p$ be the set of isomorphism classes of finite abelian $p$-groups and $\calM_{\Z_p}$ be the set of isomorphism classes of finitely generated $\Z_p$-modules.

    \item $\extz := \Z \cup \left \{ \infty \right \}$, $\Z_{\ge c} := \left \{ x \in \Z : x \ge c \right \}$ and $\extz_{\ge c} := \Z_{\ge c} \cup \left \{ \infty \right \}$ for $c \in \Z$.

    \item Let $m$ be a positive integer and $X_m = \left \{ x_1, \cdots, x_m \right \}$ be a finite ordered subset of $\Z_p$ whose elements have distinct reductions modulo $p$. 
    
    \item For $A \in \M_n(\Z_p)$, write $\cok(A) \cong \prod_{r \in \extz_{\ge 1}} (\Z_p/ p^r \Z_p)^{d_{r, A}}$ (we use the convention that $p^{\infty} = 0$ and thus $\Z_p/ p^{\infty} \Z_p = \Z_p$) and $d_{0, A} := n - \sum_{r \in \extz_{\ge 1}} d_{r, A}$. In this case, the Smith normal form of $A$ is given by 
    $$
    \diag(\overset{d_{0, A}}{\overbrace{1, \cdots, 1}}, \overset{d_{1, A}}{\overbrace{p, \cdots, p}}, \cdots, \overset{d_{\infty, A}}{\overbrace{0, \cdots, 0}}).
    $$
    
    \item For $A \in \M_n(\Z_p)$ and $k \in \Z_{\ge1}$, $\cok_{\Z_p / p^k \Z_p}(A)$ denotes the cokernel of $A/p^kA$ as a $\Z_p/p^k \Z_p$-module. It is given as $\cok_{\Z_p / p^k \Z_p}(A) \cong \prod_{r=1}^{k-1}(\Z_p/p^r\Z_p)^{d_{r, A}} \times (\Z_p/p^k\Z_p)^{\sum_{r \in \extz_{\ge k}} d_{r, A}}$. Let $\overline{A} := A/pA \in \M_n(\F_p)$ be the reduction modulo $p$ of $A$.

    \item For $A \in \M_{n \times n'}(\Z_p)$ and $k \ge 1$, denote $p^k \mid A$ if each entry of $A$ is divisible by $p^k$ and $p^k \nmid A$ otherwise. 

    \item Denote $\calB_{m}:=\{(n;H_1,\cdots,H_m)\in\Z_{\ge0}\times\calM_{\Z_p}^m:n\ge\rank_{\F_p}(H_i/pH_i)$ for every $1\le i\le m\}$. For an element $(n;H_1,\cdots,H_m) \in \calB_{m}$, write $H_i\cong\prod_{r\in\extz_{\ge1}}(\Z_p/p^r\Z_p)^{d_{r,i}}$ and $d_{0,i}:=n-\sum_{r\in\extz_{\ge1}}d_{r,i}$ for each $i$. 
    If a polynomial $P(t)\in\M_n(\Z_p)[t]$ satisfies $\cok(P(x_i))\cong H_i$ for each $i$, then we have $d_{r,P(x_i)}=d_{r,i}$ for every $r\in\extz_{\ge0}$ and $1\le i\le m$.

    \item The \textit{sum} of two elements in $\calB_{m}$ is defined by the operation
    \begin{equation} \label{eq2new1}
    (n;H_1,\cdots,H_m) + (n';H'_1,\cdots,H'_m)
    := (n+n';H_1 \times H_1', \cdots, H_m \times H_m').
    \end{equation}
\end{itemize}

\subsection{The sets \texorpdfstring{$\calC_{X_m,l,k}$}{\calC{Xmlk}}} \label{Sub21}

For $A \in \M_n(\Z_p)$ with $n_1 = n - d_{0, A}$, the Smith normal form of $A$ gives $U,V \in \GL_n(\Z_p)$ such that 
$$
UAV=\begin{pmatrix} pA' & O \\ O&I_{n-n_1} \end{pmatrix}
$$ 
for some $A' \in \M_{n_1}(\Z_p)$. For $x \in \Z_p$ and
$\displaystyle UV = \begin{pmatrix}
B_1 & B_2 \\
B_3 & B_4 \\
\end{pmatrix} \in \GL_{n_1+(n-n_1)}(\Z_p)$, we have
\begin{equation*}
\begin{split}
\cok(A+pxI)  & \cong \cok(UAV+pxUV) \\
& = \cok \left ( \begin{pmatrix}
pA'+pxB_1 & pxB_2 \\
pxB_3 & I_{n-n_1}+pxB_4 \\
\end{pmatrix} \right ) \\
& \cong \cok(p(A'+xB_1)-(pxB_2)(I_{n-n_1}+pxB_4)^{-1}(pxB_3)) \\
& = \cok \left ( p ( A' + xB_1 - \sum_{d=0}^{\infty} p^{d+1}x^{d+2}B_2(-B_4)^d B_3 ) \right ).
\end{split}    
\end{equation*}
For $A_0 = A'$, $A_1 = B_1$ and $A_r = -B_2(-B_4)^{r-2}B_3$ ($r \ge 2$), we have
\begin{equation*}
\begin{split}
\cok(A+pxI)
& \cong \cok \left ( pP^{(1)}_{A_0, A_1, \cdots}(x) \right ) \\
& := \cok \left ( p (A_0 + \sum_{d=1}^{\infty} p^{d-1}x^{d}A_{d}) \right ).
\end{split}    
\end{equation*}
In this case, we have
\begin{equation} \label{eq2a}
d_{r,A+pxI}=d_{r-1, P^{(1)}_{A_0, A_1, \cdots}(x)}
\end{equation}
for every $r\in\extz_{\ge1}$. Thus if an inequality holds for the numbers $d_{r-1,P^{(1)}_{A_0,A_1,\cdots}(x)}$ for any $A_0,A_1,\cdots\in\M_{n_1}(\Z_p)$, then the same inequality holds for the numbers $d_{r,A+pxI}$. This observation motivates us to introduce the following variant of Problem \ref{prob1f}.

\begin{definition} \label{def2a}
A polynomial in $\M_n(\Z_p)[t]$ is called an $l$\textit{-th integral of ascending polynomial}, or just an $l$\textit{-th integral} if it is of the form
$$
A_0+tA_1+\cdots+t^lA_l+pt^{l+1}A_{l+1}+\cdots+p^rt^{l+r}A_{l+r}
$$
for some $r \in \Z_{\ge 0}$ and $A_0, \cdots, A_{l+r} \in \M_n(\Z_p)$. For $l=\infty$, every polynomial in $\M_n(\Z_p)[t]$ is an $\infty$-th integral.
\end{definition}

\begin{problem} \label{prob2b}
For given $X_m$ and $l, k \in \extz_{\ge0}$, determine the set 
$$
\calC_{X_m,l,k} := \begin{Bmatrix}
(n;H_1, \cdots, H_m) \in \calB_m : \text{there exists an }l\text{-th integral } P(t) \in \M_n(\Z_p)[t]\\ \text{ such that } \cok_{\Z_p/p^k\Z_p}(P(x_i)) \cong H_i/p^kH_i \text{ for each } 1 \le i \le m
\end{Bmatrix}.
$$
\end{problem}

Now we provide basic properties of the sets $\calC_{X_m,l,k}$. For $H \in \calM_{\Z_p}$, denote $s(H) := \rank_{\F_p}(H/pH)$. Then we have $s(H_i) = \sum_{r \in \extz_{\ge 1}} d_{r,i} =  n-d_{0,i}$ for $(n; H_1, \cdots, H_m) \in \calB_m$.

\begin{proposition} \label{prop2c}
\begin{enumerate}
    \item \label{item2c1} The set $\calC_{X_m,l,k}$ is closed under the sum in $\calB_m$. In particular, $\calC_{X_m,l,k}$ is a monoid under the operation (\ref{eq2new1}) with an identity $(0;1,\cdots,1)$.
    \item \label{item2c2} For $l,l',k,k'\in\extz_{\ge0}$ and $x_0 \in \Z_p$, we have the followings:
    \begin{enumerate}
        \item \label{item2c21} $\calC_{X_m,l,k}\subset\calC_{X_m,l',k'}$ for $l\le l'$ and $k\ge k'$,

        \item \label{item2c22} $\calC_{X_m,l,\infty}=\calC_{X_m,l,m-l-1}$ for $l<m$,

        \item \label{item2c23} $\calC_{X_m,l,k}=\calC_{X_m-x_0,l,k}$ for $X_m - x_0 := \left \{ x-x_0 : x \in X_m \right \}$.
    \end{enumerate}
    
    \item \label{item2c3} {\small $\calC_{X_m}=\begin{Bmatrix}
(H_1,\cdots,H_m)\in\calM_{\Z_p}^m:\text{there exists }n\in\Z_{\ge1} \\
\text{such that } (n;H_1,\cdots,H_m)\in\calC_{X_m,0,\infty}
\end{Bmatrix}=\begin{Bmatrix}
(H_1,\cdots,H_m)\in\calM_{\Z_p}^m: s(H_1)=\cdots=s(H_m)=s\\\text{ and }(s;pH_1,\cdots,pH_m)\in\calC_{X_m,1,\infty}
\end{Bmatrix}$. }%

    \item \label{item2c4} For every $l<m$, the map 
    $$\varphi_{l,k}:
    \begin{Bmatrix}
    (n;H_1,\cdots,H_m)\in\calC_{X_m,l,k+1}: d_{0,1}=\cdots=d_{0,m}=0
    \end{Bmatrix} \rightarrow \calC_{X_m,l+1,k}
    $$
    ($(n;H_1,\cdots,H_m) \mapsto (n;pH_1,\cdots,pH_m)$) is well-defined and a bijection.
\end{enumerate}

\begin{proof}
\begin{enumerate}
    \item[(1)] For every $(n;H_1,\cdots,H_m),(n';H'_1,\cdots,H'_m) \in \calC_{X_m,l,k}$, there are $l$-th integrals $P(t)\in\M_n(\Z_p)[t]$ and $P'(t)\in\M_{n'}(\Z_p)[t]$ such that $\cok_{\Z_p/p^k\Z_p}(P(x_i))\cong H_i$ and $\cok_{\Z_p/p^k\Z_p}(P'(x_i))\cong H'_i$ for each $i$. Then the \textit{concatenation} of $P(t)$ and $P'(t)$ given by
    $$
    Q(t) := \begin{pmatrix}P(t)&O\\O&P'(t)\end{pmatrix} \in \M_{n+n'}(\Z_p)[t]
    $$ 
    is also an $l$-th integral and $\cok_{\Z_p/p^k\Z_p}(Q(x_i)) \cong H_i \times H_i'$ for each $i$. Thus we have
$$
(n+n'; H_1 \times H_1', \cdots, H_m \times H_m') \in \calC_{X_m,l,k}.
$$
    \item[(2a)] It follows from the facts that an $l$-th integral is also an $l'$-th integral and 
    $$
    \cok_{\Z_p/p^{k'}\Z_p}(A) \cong \cok_{\Z_p/p^{k}\Z_p}(A)/p^{k'}\cok_{\Z_p/p^{k}\Z_p}(A).
    $$
    
    \item[(2b)] The inclusion $\subset$ holds by (\ref{item2c21}). Now suppose that $(n;H_1,\cdots,H_m)\in\calC_{X_m,l,m-l-1}$. Then there exists an $l$-th integral $P(t)\in\M_n(\Z_p)$ such that $\cok_{\Z_p/p^{m-l-1}\Z_p}(P(x_i))\cong H_i/p^{m-l-1}H_i$ for each $i$. The Smith normal form of $P(x_i)$ gives $U_i,V_i\in\GL_n(\Z_p)$ such that 
    $$
    P(x_i)=U_i \diag(\overset{d_{0, P(x_i)}}{\overbrace{1, \cdots, 1}}, \cdots, \overset{d_{m-l-1, P(x_i)}}{\overbrace{p^{m-l-1}, \cdots, p^{m-l-1}}}, \, \overset{d_{m-l, P(x_i)}}{\overbrace{p^{m-l}, \cdots, p^{m-l}}}, \cdots, \overset{d_{\infty, P(x_i)}}{\overbrace{0, \cdots, 0}}) V_i.
    $$
    Define $\displaystyle L_{m,i}(t):=\prod_{1\le j \le m, j \neq i} \frac{t-x_j}{x_i-x_j} \in \Z_p[t]$, $D_{m-l-1,i}:=\sum_{r\in\extz_{\ge m-l-1}}d_{r,P(x_i)}$ and
    \begin{equation*}
    Q(t) := P(t)+\sum_{i=1}^mL_{m,i}(t) U_i \diag(\overset{n-D_{m-l-1,i}}{\overbrace{0, \cdots, 0}}, p^{a_{i,1}}-b_{i,1},\cdots,p^{a_{i,D_{m-l-1,i}}}-b_{i,D_{m-l-1,i}}) V_i
    \end{equation*}
    where $(b_{i,1}, \cdots, b_{i,D_{m-l-1, i}}) := (\overset{d_{m-l-1, P(x_i)}}{\overbrace{p^{m-l-1}, \cdots, p^{m-l-1}}}, \cdots, \overset{d_{\infty, P(x_i)}}{\overbrace{0, \cdots, 0}})$ and $a_{i,j} \in \extz_{\ge m-l-1}$ for $1\le i\le m$ and $1\le j\le D_{m-l-1,i}$. Then $Q(t)$ is also an $l$-th integral as $p^{m-l-1} \mid p^{a_{i,j}}-b_{i,j}$ for each $i, j$ and $L_{m,i}(t)$ is of degree $m-1$, while
    \begin{equation*}
    \begin{split}
    Q(x_i)&=P(x_i)+U_i \diag(\overset{n-D_{m-l-1,i}}{\overbrace{0, \cdots, 0}}, p^{a_{i,1}}-b_{i,1},\cdots,p^{a_{i,D_{m-l-1,i}}}-b_{i,D_{m-l-1,i}}) V_i \\
    &=U_i \diag(\overset{d_{0, P(x_i)}}{\overbrace{1, \cdots, 1}}, \cdots, \overset{d_{m-l-2, P(x_i)}}{\overbrace{p^{m-l-2}, \cdots, p^{m-l-2}}}, p^{a_{i,1}},\cdots,p^{a_{i,D_{m-l-1,i}}})V_i
    \end{split}
    \end{equation*}
    for each $i$. Now we can choose $a_{i,j} \in \extz_{\ge m-l-1}$ ($1\le i\le m$, $1\le j\le D_{m-l-1,i}$) such that $\cok(Q(x_i))\cong H_i$ for each $i$, which implies that $(n;H_1,\cdots,H_m)\in\calC_{X_m,l,\infty}$.

    \item[(2c)] $P(t) \in \M_n(\Z_p)[t]$ is an $l$-th polynomial if and only if $P_1(t) := P(t+x_0) \in \M_n(\Z_p)[t]$ is an $l$-th polynomial and $\cok_{\Z_p/p^k\Z_p}(P(x_i)) \cong \cok_{\Z_p/p^k\Z_p}(P_1(x_i-x_0))$ so we have $\calC_{X_m,l,k} = \calC_{X_m-x_0,l,k}$.
    
    \item[(3)] Consider the sets
    \begin{align*}
     S_1 & := \begin{Bmatrix}
(H_1,\cdots,H_m)\in\calM_{\Z_p}^m: \exists n \in \Z_{\ge 1} \\
\text{s.t. } (n;H_1,\cdots,H_m)\in\calC_{X_m,0,\infty}
\end{Bmatrix},  \\
S_2 & := \begin{Bmatrix}
(H_1,\cdots,H_m)\in\calM_{\Z_p}^m: s(H_1)=\cdots=s(H_m)=s\\\text{ and }(s;pH_1,\cdots,pH_m)\in\calC_{X_m,1,\infty}
\end{Bmatrix}.
    \end{align*}

    \begin{itemize}
        \item ($\calC_{X_m}=S_1$) The inclusion $\subset$ follows from the definition. Suppose that $(n;H_1,\cdots,H_m) \in \calC_{X_m,0,\infty}$ so that there exists a zeroth integral $P(t)\in\M_n(\Z_p)[t]$ such that $\cok(P(x_i))\cong H_i$ for each $i$. Let
        \begin{equation*}
        \begin{split}
        Q(t) & :=P(t)+p^{d+m}t^{d}(t-x_1)\cdots(t-x_m)I_n \\ &=A_0+ptA_1+\cdots+p^{d+m-1}t^{d+m-1}A_{d+m-1}+p^{d+m}t^{d+m}I_n
        \end{split}
        \end{equation*}
        for $d = \deg(P)$ and $A_0,\cdots,A_{d+m-1}\in\M_n(\Z_p)$. The rational canonical form of $Q(t)$, i.e. 
        $$
        A:=\begin{pmatrix}&&&A_0\\-I_n&&&A_1\\&\ddots&&\vdots\\&&-I_n&A_{d+m-1}\end{pmatrix} \in \M_{(d+m)n}(\Z_p)
        $$
        satisfies $\cok(A+px_iI)\cong \cok(Q(x_i)) \cong \cok(P(x_i)) \cong H_i$ for each $i$ so we have $(H_1,\cdots,H_m) \in \calC_{X_m}$.

        \item ($\calC_{X_m}\subset S_2$) Suppose that $(H_1,\cdots,H_m)\in\calC_{X_m}$ so that there exists $A\in\M_n(\Z_p)$ for some $n \ge 1$ such that $\cok(A+px_iI)\cong H_i$ for each $i$. Then $s(H_i) = \rank_{\F_p}(H_i/pH_i) = \rank_{\F_p}(\cok(\overline{A}))$ so we have $s(H_1) = \cdots = s(H_m)=s$. By the equation (\ref{eq2a}), there exists a first integral $P(t) \in \M_s(\Z_p)[t]$ such that $\cok(P(x_i)) \cong pH_i$ for each $i$, which implies that $(s;pH_1,\cdots,pH_m)\in\calC_{X_m,1,\infty}$.

        \item ($S_2\subset S_1$) Suppose that $(H_1,\cdots,H_m)\in S_2$ so that $s(H_1)=\cdots=s(H_m)=s$ and $(s;pH_1,\cdots,pH_m)\in\calC_{X_m,1,\infty}$. Let $P(t)\in\M_s(\Z_p)[t]$ be a first integral such that $\cok(P(x_i))\cong pH_i$ for each $i$. Since $p P(t)$ is a zeroth integral and $\cok(p P(x_i)) \cong H_i$ for each $i$, we have $(s;H_1,\cdots,H_m)\in\calC_{X_m,0,\infty}$.
    \end{itemize}

     \item[(4)] Assume that $(n;H_1,\cdots,H_m)\in\calC_{X_m,l,k+1}$ with $d_{0,1}=\cdots=d_{0,m}=0$. Then there exists an $l$-th integral $P(t) \in \M_n(\Z_p)[t]$ such that $\cok_{\Z_p/p^{k+1}\Z_p}(P(x_i))\cong H_i/p^{k+1}H_i$ and $d_{0,P(x_i)}=0$ (so $p \mid P(x_i)$) for each $i$. This implies that $P(t) = (t-x_1)\cdots(t-x_m)Q(t)+pP_1(t)$ for some $Q(t), P_1(t) \in \M_n(\Z_p)[t]$ such that $\deg((t-x_1)\cdots(t-x_m)Q(t)) \le l$ and $P_1(t)$ is an $(l+1)$-th integral. 
     Since we have $P(x_i) = pP_1(x_i)$ for each $i$, we have $d_{r,P_1(t)}=d_{r+1,P(t)}$ for every $r\in\extz_{\ge1}$ so $\cok_{\Z_p/p^k\Z_p}(P_1(x_i))=pH_i/p^{k+1}H_i$ for each $i$. Thus the map $\varphi_{l,k}$ is well-defined.

    Now assume that $(n;H_1,\cdots,H_m)\in\calC_{X_m,l+1,k}$. Then there exists an $(l+1)$-th integral $P(t)\in\M_n(\Z_p)[t]$ such that $\cok_{\Z_p/p^k\Z_p}(P(x_i))\cong H_i/p^kH_i$ for each $i$. Since $pP(t)$ is an $l$-th integral which satisfies $d_{0,p P(t)}=0$ and $d_{r,P(t)}=d_{r+1,p P(t)}$ for every $r\in\extz_{\ge0}$, the map
    $$\psi_{l,k} :\calC_{X_m,l+1,k} \rightarrow
    \begin{Bmatrix}
    (n;H_1,\cdots,H_m)\in\calC_{X_m,l,k+1}: d_{0,1}=\cdots=d_{0,m}=0
    \end{Bmatrix}
    $$
    given by $(n;H_1,\cdots,H_m) \mapsto (n;\cok(p P(x_1)),\cdots,\cok(p P(x_m)))$ is the inverse of $\varphi_{l,k}$. \qedhere
\end{enumerate}
\end{proof}
\end{proposition}

As an example, we determine the elements of the set $\calC_{X_m,l,\infty}$ which are of the form $(1; H_1, \cdots, H_m)$. This result will be frequently used in the proof of Theorem \ref{thm1h}. We note that if $(n;H_1,\cdots,H_m)\in\calC_{X_m,l,k}$, then $(n;H_{\sigma(1)},\cdots,H_{\sigma (m)})\in\calC_{X_m,l,k}$ for any permutation $\sigma \in S_m$.

\begin{example} \label{ex2d} 
For an element $(1; H_1, \cdots, H_m) \in \calC_{X_m,l,\infty}$, there exists an $l$-th integral $P(t) \in \Z_p[t]$ such that $\Z_p/P(x_i)\Z_p \cong H_i$ for each $i$. If $P(t)=0$, then $H_1 = \cdots = H_m = \Z_p$. If $P(t)$ is not identically zero, then there uniquely exists $r \in \Z_{\ge 0}$ such that $P(t) = p^r Q(t)$ for some $(l+r)$-th integral $Q(t) \in \Z_p[t]$ whose reduction modulo $p$ is not identically zero in $\F_p[t]$. Since $Q(t) \equiv 0 \,\, (\bmod \; p)$ has at most $l+r$ roots modulo $p$, the number of $1 \le i \le m$ such that $H_i = \Z_p/p^r\Z_p$ is at least $m-(l+r)$. Conversely, for every integer $0 \le r \le m-l$ and $b_1, \cdots, b_{r'} \in \extz_{\ge r+1}$ with $r'\le l+r$, an $l$-th integral
$$
P(t) = p^r\prod_{i=1}^{r'} (t-x_i) + \sum_{i=1}^{r'} p^{b_i} L_{r',i}(t)\in \Z_p[t]
$$
satisfies $\Z_p/P(x_i)\Z_p \cong \Z_p/p^{b_i}\Z_p$ for $1 \le i \le r'$ and $\Z_p/P(x_i)\Z_p \cong \Z_p/p^r\Z_p$ for $i>l+r$. (The polynomials $L_{m,i}(t)$ are defined as in the proof of Proposition \ref{prop2c}(\ref{item2c22}).) 

Now we deduce that $(1; H_1, \cdots, H_m) \in \calC_{X_m,l,\infty}$ if and only if $(H_1, \cdots, H_m)$ is a permutation of
$$
(\Z_p/p^{b_1}\Z_p,\cdots,\Z_p/p^{b_{r+l}}\Z_p,\overbrace{\Z_p/p^r\Z_p,\cdots,\Z_p/p^r\Z_p}^{m-r-l})
$$
for some $0 \le r \le m-l$ and $b_1, \cdots, b_{r+l}\in\extz_{\ge r}$. In particular, we have 
$$
(\Z_p/p^{b_1}\Z_p,\cdots,\Z_p/p^{b_{r}}\Z_p,\overbrace{\Z_p/p^r\Z_p,\cdots,\Z_p/p^r\Z_p}^{m-r}) \in \calC_{X_m}
$$
for every $0 \le r \le m$ and $b_1, \cdots, b_r \in \extz_{\ge r}$ by Proposition \ref{prop2c}(\ref{item2c3}).
\end{example}
\subsection{Proof of Theorem \ref{thm1h}: the case \texorpdfstring{$m \le 3$}{m le 3}} \label{Sub22}

The case $m=1$ is trivial. When $m=2$, we have $\calC_{X_2} \subset \left \{ (H_1, H_2) \in \calM_{\Z_p}^2 : s_1=s_2 \right \}$ (see the paragraph after Theorem \ref{thm1h}). Conversely, every element $(H_1, H_2) \in \calM_{\Z_p}^2$ such that $s_1=s_2=s$ is of the form 
$$
\left ( \prod_{j=1}^{s} \Z_p/p^{a_j}\Z_p, \, \prod_{j=1}^{s} \Z_p/p^{b_j}\Z_p \right ) \,\, (a_j, b_j \in \extz_{\ge 1}).
$$
Since the set $\calC_{X_2}$ is closed under finite direct product, to prove Theorem \ref{thm1h} for $m=2$ it is enough to show that $(\Z_p/p^a\Z_p, \Z_p/p^b\Z_p) \in \calC_{X_2}$ for every $a, b \in \extz_{\ge 1}$. We already proved this in Example \ref{ex2d}.

Now we consider the case $m=3$. First we prove that every element $(H_1,H_2,H_3) \in \calC_{X_3}$ satisfies the condition $2d_{1,i} \le D_1$. By the equation (\ref{eq2a}), it is enough to show that the numbers $d_{0, P^{(1)}_{A_0, A_1, \cdots}(x_i)}$ ($1 \le i \le 3$) satisfy the triangle inequality for every $A_0, A_1, \cdots\in\M_{n_1}(\Z_p)$. The congruence $P^{(1)}_{A_0, A_1, \cdots}(x_i) \equiv A_0+x_iA_1 \,\, (\bmod \; p)$ implies that $d_{0, P^{(1)}_{A_0, A_1, \cdots}(x_i)} = d_{0, A_0+x_iA_1} = \dim_{\F_p} N(\overline{A_0+x_iA_1})$, where $N(\overline{A_0+x_iA_1})$ denotes the null space of $\overline{A_0+x_iA_1} \in \M_n(\F_p)$. By the relation
$$
(x_2-x_3)(A_0+x_1A_1) + (x_3-x_1)(A_0+x_2A_1) + (x_1-x_2)(A_0+x_3A_1) = O,
$$
the numbers $\dim_{\F_p} N(\overline{A_0+x_iA_1})$ ($1 \le i \le 3$) satisfy the triangle inequality. We conclude that 
$$
\calC_{X_3} \subset \left \{ (H_1,H_2,H_3) \in \calM_{\Z_p}^3 : s_1=s_2=s_3 \text{ and } 2d_{1, i} \le D_1 \;\; (1 \le i \le 3) \right \}.
$$

It remains to show that every $(H_1,H_2,H_3) \in \calM_{\Z_p}^3$ satisfying the conditions $s_1=s_2=s_3=s$ and $2d_{1, i} \le D_1$ ($1 \le i \le 3$) is an element of $\calC_{X_3}$. For each $i$, let $c_i:=d_{1,i}$ and $H_i \cong (\Z_p/p\Z_p)^{c_i} \times \prod_{j=c_i+1}^{s} \Z_p / p^{b_{i, j}} \Z_p$ for some $b_{i, c_i+1}, \cdots, b_{i, s} \in \extz_{\ge 2}$. We may assume that $c_1 =\underset{1 \le i \le 3}{\max} c_i$. Then $c_2+c_3-c_1 \ge 0$ and
\begin{equation} \label{eq2b}
\begin{split}
(H_1,H_2,H_3) & \cong (\Z_p/p\Z_p, \Z_p/p\Z_p, \Z_p/p\Z_p)^{c_2+c_3-c_1} \\
& \times \prod_{j=1}^{c_1-c_3} (\Z_p/p\Z_p, \Z_p/p\Z_p, \Z_p/p^{b_{3, j+c_3}}\Z_p) \\
& \times \prod_{j=1}^{c_1-c_2} (\Z_p/p\Z_p, \Z_p/p^{b_{2, j+c_2}}\Z_p, \Z_p/p\Z_p) \\
& \times \prod_{j=c_1+1}^{s} (\Z_p/p^{b_{1,j}}\Z_p, \Z_p/p^{b_{2,j}}\Z_p, \Z_p/p^{b_{3,j}}\Z_p).
\end{split}    
\end{equation}
Each term in the right-hand side of the equation (\ref{eq2b}) is contained in $\calC_{X_3}$ by Example \ref{ex2d} and the set $\calC_{X_3}$ is closed under finite direct product, we conclude that $(H_1,H_2,H_3) \in \calC_{X_3}$. This finishes the proof of Theorem \ref{thm1h} for $m \le 3$.

\section{Proof of Theorem \ref{thm1h} for \texorpdfstring{$m=4$}{m=4}} \label{Sec3}

In this section, we prove Theorem \ref{thm1h} for $m=4$. First we prove a necessary condition for an element of $\calC_{X_4}$ using a reduction procedure. The purpose of a reduction procedure is to reduce the size of a matrix $n$ without information loss of $(H_1,\cdots,H_m)$ and to extract inequalities using Proposition \ref{prop2c}. After that, we prove that the same condition is also a sufficient condition for an element of $\calC_{X_4}$ using zone theory. Throughout this section, we assume that $x_1 = 0$. (We may assume this by Proposition \ref{prop2c}(\ref{item2c23})).

\subsection{A reduction procedure} \label{Sub31}
We begin by clarifying the relation between zeroth and first integrals. Recall that if $(n;H_1,\cdots,H_m)\in\calC_{X_m,0,k}$ for $k\in\extz_{\ge1}$, then $d_{0,1}=\cdots=d_{0,m}$ where $d_{0,i} = n-\sum_{r\in\extz_{\ge1}}d_{r,i}$ for each $i$.

\begin{proposition} \label{prop3a}
If $(n;H_1,\cdots,H_m)\in\calC_{X_m,0,k}$ and $d_{0,1}=\cdots=d_{0,m} > 0$, then $(n-1;H_1,\cdots,H_m)\in\calC_{X_m,0,k}$.
\end{proposition}

\begin{proof}
Since we have $\calC_{X_m,0,\infty}=\calC_{X_m,0,m-1}$, we may assume that $k$ is finite. Let $P(t)\in\M_n(\Z_p)[t]$ be a zeroth integral which satisfies $\cok_{\Z_p/p^k\Z_p}(P(x_i))\cong H_i/p^kH_i$ for each $i$. The constant term of $P(t)$ satisfies $\rank_{\F_p}(P(0))=d_{0,1}>0$. Using the Smith normal form of $P(0)$, we may assume that 
\begin{align*}
P(t)&=\begin{pmatrix}A_0&O\\O&1\end{pmatrix}+ptA_1+\cdots+p^rt^rA_r\\&=:\begin{pmatrix}P_1(t)&ptf(t)\\ptg(t)&1+pth(t)\end{pmatrix}\in\M_{(n-1)+1}(\Z_p)[t]
\end{align*}
for some $r\in\Z_{\ge0}$ and zeroth integrals $f(t),g(t),h(t),P_1(t)$. Then we have
\begin{align*}
\cok_{\Z_p/p^k\Z_p}(P(t))&=\cok_{\Z_p/p^k\Z_p}\left(\begin{pmatrix}P_1(t)&ptf(t)\\ptg(t)&1+pth(t)\end{pmatrix}\right)\\&\cong\cok_{\Z_p/p^k\Z_p}\left(\begin{pmatrix}P_1(t)-p^2t^2(1+pth(t))^{-1} f(t)g(t)&O\\O&1+pth(t)\end{pmatrix}\right)\\&\cong\cok_{\Z_p/p^k\Z_p} (Q(t))
\end{align*}
where $Q(t):=P_1(t)-p^2t^2(\sum_{j=0}^{k-1}(-1)^jp^jt^jh(t)^j) f(t)g(t)\in\M_{n-1}(\Z_p)[t]$ is a zeroth integral as the set of zeroth integrals is closed under sum and product. This implies that $(n-1;H_1,\cdots,H_m)\in\calC_{X_m,0,k}$.
\end{proof}

Recall that the set $\calC_{X_m,l,k}$ has a monoid structure by Proposition \ref{prop2c}(\ref{item2c1}). For $S \subset \calC_{X_m,l,k}$, let $\langle S \rangle$ be the submonoid of $\calC_{X_m,l,k}$ generated by the elements of $S$.

\begin{corollary} \label{cor3b}
For every $k \in \extz_{\ge 0}$, we have $\calC_{X_m,0,k+1}=\langle \left \{ (1;1,\cdots,1) \right \} \cup \varphi^{-1}_{0,k}(\calC_{X_m,1,k}) \rangle$.
\end{corollary}

\begin{proof}
The inclusion $\supset$ is clear since a zeroth integral $P(t)=1$ gives $(1;1,\cdots,1) \in \calC_{X_m,0,k+1}$. Conversely, every element $(n; H_1, \cdots, H_m) \in \calC_{X_m,0,k+1}$ satisfies $n \ge s = s(H_1) = \cdots = s(H_m)$ so
$$
(n; H_1, \cdots, H_m) = (n-s)(1; 1, \cdots, 1) + (s; H_1, \cdots, H_m) \in \langle \left \{ (1;1,\cdots,1) \right \} \cup \varphi^{-1}_{0,k}(\calC_{X_m,1,k}) \rangle
$$
by Proposition \ref{prop3a} and \ref{prop2c}(\ref{item2c4}).
\end{proof}

By Proposition \ref{prop2c}(\ref{item2c22}) and \ref{prop2c}(\ref{item2c3}), the set $\calC_{X_m}$ is determined by the set $\calC_{X_m,0,m-1}$, which is determined by the set $\calC_{X_m,1,m-2}$ by Corollary \ref{cor3b}. Since we have $\calC_{X_m,1,m-2} \subset \calC_{X_m,1,1}$ for every $m \ge 3$, an inequality which holds for elements of $\calC_{X_m,1,1}$ also holds for elements of $\calC_{X_m,1,m-2}$. 


\begin{lemma} \label{lem3c}
If $(n;H_1,\cdots,H_m)\in\calC_{X_m,1,1}$, then
$$
\sum_{i=1}^{m}d_{0, i} - (m-1)\alpha_0\ge0
$$
for some non-negative integer $\alpha_0\ge\underset{1\le i\le m}{\max}d_{0,i}$.
\end{lemma}

\begin{proof}
It suffices to show that $\sum_{i=1}^{m}d_{0, i} \ge (m-1)\underset{1\le i\le m}{\max}d_{0,i}$. We phrased the result in this way to make it consistent with Theorem \ref{thm3g}.

We use induction on $n$. The case $n=1$ follows from Example \ref{ex2d}. Now assume that $n > 1$ and the theorem holds for every $n' < n$. Suppose that there exists $(n;H_1,\cdots,H_m)\in\calC_{X_m,1,1}$ such that
$$
\sum_{i=1}^{m} d_{0,i} < (m-1) \underset{1\le i\le m}{\max}d_{0,i}.
$$
Let $P(t)\in\M_n(\Z_p)[t]$ be a first integral such that $\cok_{\Z_p/p\Z_p}(P(x_i))\cong H_i/pH_i$ for each $i$. Applying the Smith normal form of the constant term of $P(t)$, we may assume that $P(t) = \begin{pmatrix}
pA & O\\ 
O & I_{d_{0,1}}
\end{pmatrix} + tB + pQ(t)$ for some $A \in \M_{n-d_{0,1}}(\Z_p)$, $B \in \M_n(\Z_p)$ and $Q(t) \in \M_n(\Z_p)[t]$. Moreover, $P_1(t) := \begin{pmatrix}
O & O\\ 
O & I_{d_{0,1}}
\end{pmatrix} + tB$ satisfies $d_{0, P(x_i)} = d_{0, P_1(x_i)}$ for each $i$ so we may assume that
$$
P(t) = \begin{pmatrix} O&O\\O&I_{d_{0,1}} \end{pmatrix}+ t \begin{pmatrix}B_1&B_2\\B_3&B_4\end{pmatrix}.
$$

\phantom{}

\noindent \textbf{Case 1.} $p\nmid B_1$. 

The Smith normal form of $B_1$ gives $U,V\in \GL_{n-d_{0,1}}(\Z_p)$ such that $UB_1V=\begin{pmatrix}1&O\\O&B_1'\end{pmatrix}$ for some $B_1'\in\M_{n-d_{0,1}-1}(\Z_p)$. For every non-zero $x\in X_m$, we have
\begin{align*}
\cok_{\Z_p/p\Z_p}(P(x))
&\cong \cok_{\Z_p/p\Z_p}\left(\begin{pmatrix}U&O\\O&I_{d_{0,1}}\end{pmatrix} P(x)\begin{pmatrix}V&O\\O&I_{d_{0,1}}\end{pmatrix}\right)\\
&=\cok_{\Z_p/p\Z_p}\left(\begin{pmatrix}\begin{matrix}x&O\\O&xB_1'\end{matrix}&xUB_2\\xB_3V&I_{d_{0,1}}+xB_4\end{pmatrix}\right)\\
& \cong\cok_{\Z_p/p\Z_p} \left(\begin{pmatrix}x&O \\O& Q(x) \end{pmatrix}\right)
\end{align*}
for a first integral $Q(t) := \begin{pmatrix}O&O\\O&I_{d_{0,1}}\end{pmatrix} + t \begin{pmatrix}B_1'&B_2'\\B_3'&B_4'\end{pmatrix}\in\M_{n-1}(\Z_p)[t]$. 
Then $d_{0,Q(x_1)}=d_{0,1}$, $d_{0,Q(x_i)}=d_{0,i}-1$ for $2 \le i \le m$ and $(n-1;\cok(Q(x_1)),\cdots,\cok(Q(x_m)))$ contradicts the induction hypothesis since 
\begin{equation*}
\begin{split}
\sum_{i=1}^{m} d_{0,Q(x_i)} -(m-1) \underset{1\le i\le m}{\max}d_{0,Q(x_i)}
& = \sum_{i=1}^{m} d_{0,i} -(m-1)(\underset{1\le i\le m}{\max}d_{0,Q(x_i)}+1) \\
& \le \sum_{i=1}^{m} d_{0,i} -(m-1) \underset{1\le i\le m}{\max}d_{0,i} \\
& <0.
\end{split}    
\end{equation*}

\phantom{}

\noindent \textbf{Case 2.} $p \mid B_1$ and [$p\nmid B_2$ or $p\nmid B_3$].

We may assume that $p \nmid B_2$. Choose invertible matrices $U$ and $V$ such that
$UB_2V=\begin{pmatrix}1&O\\O&B_2'\end{pmatrix}$. For every non-zero $x\in X_m$, we have
\begin{align*}
\cok_{\Z_p/p\Z_p}(P(x))&\cong \cok_{\Z_p/p\Z_p}\left(\begin{pmatrix}U&O\\O&V^{-1}\end{pmatrix}P(x)\begin{pmatrix}I_{d_{1,1}}&O\\O&V\end{pmatrix}\right)\\
&=\cok_{\Z_p/p\Z_p}\left(\begin{pmatrix}O&\begin{matrix}x&O\\O&xB_2'\end{matrix}\\xV^{-1}B_3&I_{d_{0,1}}+xV^{-1}B_4V\end{pmatrix}\right)\\
&\cong\cok_{\Z_p/p\Z_p} \left(\begin{pmatrix}x&O \\O& Q(x) \end{pmatrix}\right)
\end{align*}
for a first integral $Q(t) := \begin{pmatrix} O&O\\O&I_{d_{0,1}-1} \end{pmatrix} + t \begin{pmatrix}O&B_2'\\B_3'&B_4'\end{pmatrix}\in\M_{n-1}(\Z_p)[t]$. 
Then $d_{0,Q(x_i)}=d_{0,i}-1$ for $1 \le i \le m$ and $(n-1;\cok(Q(x_1)),\cdots,\cok(Q(x_m)))$ contradicts the induction hypothesis since 
\begin{equation*}
\sum_{i=1}^{m} d_{0,Q(x_i)} -(m-1) \underset{1\le i\le m}{\max}d_{0,Q(x_i)}
= \sum_{i=1}^{m} d_{0,i} -(m-1) \underset{1\le i\le m}{\max}d_{0,i} -1 < 0.
\end{equation*}

\phantom{ }

In the remaining cases, we have $p \mid B_1,B_2,B_3$ so that $d_{0,1} = \underset{1\le i\le m}{\max} d_{0, i}$. By the same reason, we have $d_{0,1} = \cdots = d_{0,m}$. Then we have $\sum_{i=1}^{m} d_{0,i} -(m-1) \underset{1\le i\le m}{\max}d_{0,i} = d_{0,1} \ge 0$, a contradiction.
\end{proof}

We need some work to find the conditions for elements of $\calC_{X_m,1,2}$.

\begin{lemma} \label{lem3d}
Assume that $J\in\M_n(\Z_p)$ has a single non-zero row or column. For every $r\in\Z_{\ge 1}$, $x \in \Z_p$  and $A\in\M_n(\Z_p)$, there exists $x_0\in \Z_p$ and $d\in\Z_{\ge0}$ such that 
$$
d_{i,A}=d_{i,A+p^rxJ}\text{ for }i<r\text{ and }d_{r,A+p^rxJ}=\begin{cases}d-1\text{ or }d & (x \equiv x_0 \,\, (\bmod \; p)) \\ 
d & (x \not\equiv x_0 \,\, (\bmod \; p))\end{cases}.
$$
$J$ is called \textit{singular} of order $r$ at $x_0$ over $A$ if $d_{r,A+p^rx_0J}=d-1$.
\end{lemma}

\begin{proof}
We may assume that $J=\begin{pmatrix} J_1 & O_{n \times (n-1)} \end{pmatrix}$. The isomorphism $\cok_{\Z_p/p^r\Z_p}(A) \cong \cok_{\Z_p/p^r\Z_p}(A+p^rxJ)$ implies that $d_{i,A}=d_{i,A+p^rxJ}$ for each $i < r$. 
Let $A_j$ be the $j$-th column of $A$ for each $j$ and $A_{s_1},\cdots,A_{s_D}$ ($2 \le s_j \le m$) be any maximal subset of $\left \{ A_2, \cdots, A_m \right \}$ whose elements are linearly independent modulo $p^{r+1}$, i.e. $\sum_{j=1}^{D} c_jA_{s_j} \equiv 0 \,\, (\bmod \; p^{r+1})$ implies that $c_1, \cdots, c_{D}  \equiv 0 \,\, (\bmod \; p)$. (For example, $(1, 1), (1, p+1) \in \Z_p^2$ are not linearly independent modulo $p$ but are linearly independent modulo $p^2$.) Define
$$
S:=\begin{Bmatrix}
v\in\Z_p^n : \text{there exist }c_0,\cdots,c_D\in\Z_p\text{ such that }p\nmid c_j\text{ for some } \\ 0\le j\le D \text{ and } c_0v+c_1A_{s_1}+\cdots+c_DA_{s_D} \equiv 0 \,\, (\bmod \; p^{r+1})
\end{Bmatrix}.
$$ 
Since the number $\sum_{i=0}^rd_{i,A}$ is the maximum number of linearly independent columns of $A$ modulo $p^{r+1}$, we have
$$
D_x := \sum_{i=0}^rd_{i,A+p^rxJ}=\begin{cases}D &\text{if }A_1+p^rx J_1\in S,\\
D+1 &\text{otherwise}\end{cases}.
$$
If we have $D_x = D+1$ for every $x \in \Z_p$, then the lemma holds for $d = d_{r, A}+1$ and any $x_0 \in \Z_p$. Now assume that $D_{x_0}=D$ for some $x_0 \in \Z_p$. Then there exist $c_0,\cdots,c_D\in\Z_p$ such that $p\nmid c_j$ for some $j$ and 
$$
c_0(A_1+p^r{x_0}J_1)+c_1A_{s_1}+\cdots+c_DA_{s_D} \equiv 0 \,\, (\bmod \; p^{r+1}).
$$ 
If $p\mid c_0$, then $c_0(A_1+p^rx_0J_1) \equiv c_0(A_1+p^rxJ_1) \,\, (\bmod \; p^{r+1})$ for every $x \in \Z_p$ so we have $D_{x}=D$ for every $x \in \Z_p$. If $p\nmid c_0$, then $A_1+p^rx_0J_1$ is a $\Z_p$-linear combination of $A_{s_1}, \cdots, A_{s_D}$ modulo $p^{r+1}$. If there exists $x_1 \not \equiv x_0 \,\, (\bmod \; p)$ such that $A_1+p^rx_1J_1$ is a $\Z_p$-linear combination of $A_{s_1}, \cdots, A_{s_D}$ modulo $p^{r+1}$, then
$$
\frac{(A_1+p^r{x_0}J_1)-(A_1+p^r{x_1}J_1)}{x_0-x_1}=p^rJ_1
$$
is also a $\Z_p$-linear combination of $A_{s_1}, \cdots, A_{s_D}$ modulo $p^{r+1}$ so we have $D_x = D$ for every $x \in \Z_p$.
If there is no such $x_1$, then $D_x = D+1$ if and only if $x \not \equiv x_0 \,\, (\bmod \; p)$. 
\end{proof}

The following lemma is the most technical part of this paper.

\begin{lemma} \label{lem3e}
Let $P(t)\in\M_n(\Z_p)[t]$ be a first integral. Then either one of the following holds:
\begin{enumerate}
    \item \label{item3d1} There exists a first integral $Q(t) \in \M_{n-1}(\Z_p)[t]$ satisfies either, for at most one $1\le r\le m$,
    \begin{enumerate}
        \item \label{item3d11} $\begin{cases}d_{0,Q(x_i)}=d_{0,P(x_i)}-1\text{ for every }1\le i\le m, \\ 
        d_{1,Q(x_i)}\in\{d_{1,P(x_i)},d_{1,P(x_i)}+1\}\text{ for }i=r\text{ and }d_{1,Q(x_i)}=d_{1,P(x_i)}\text{ otherwise;}\end{cases}$
        \item \label{item3d12} $\begin{cases}d_{0,Q(x_i)}=d_{0,P(x_i)}\text{ for }i=r\text{ and }d_{0,Q(x_i)}=d_{0,P(x_i)}-1\text{ otherwise,}\\ 
        d_{1,Q(x_i)}\in \left \{ d_{1,P(x_i)}-1, d_{1,P(x_i)}-2 \right \}\text{ for }i=r\text{ and }d_{1,Q(x_i)}=d_{1,P(x_i)}\text{ otherwise;}\end{cases}$
        \item \label{item3d13}  $\begin{cases}d_{0,Q(x_i)}=d_{0,P(x_i)}\text{ for }i=r\text{ and }d_{0,Q(x_i)}=d_{0,P(x_i)}-1\text{ otherwise,}\\ 
        d_{1,Q(x_i)}=d_{1,P(x_i)}\text{ for every }1\le i\le m. \end{cases}$
    \end{enumerate}

    \item \label{item3d2} The number $d_{0,P(x_i)}$ is constant for every $1 \le i \le m$ and there exists a first integral $Q(t) \in \M_{n-d_{0,P(x_i)}}(\Z_p)[t]$ satisfying $d_{0,Q(x_i)}=0$ and $d_{1,Q(x_i)}=d_{1,P(x_i)}$ for each $i$.
\end{enumerate}
\end{lemma}
\begin{proof}
As in the proof of Lemma \ref{lem3c}, we may assume that 
$$
P(t):=\begin{pmatrix}O&O&O\\O&pI_{d_{1,1}}&O\\O&O&I_{d_{0,1}}\end{pmatrix}+t\begin{pmatrix}B_{11}&B_{12}&B_{13}\\B_{21}&B_{22}&B_{23}\\B_{31}&B_{32}&B_{33}\end{pmatrix}+pt^2C \in \M_{(n-d_{0,1}-d_{1,1})+d_{1,1}+d_{0,1}}(\Z_p).
$$
Denote $d_{2,1}' := \sum_{r \in \extz_{\ge 2}} d_{r,1} = n-(d_{0,1}+d_{1,1})$ for simplicity.

\phantom{ }

\noindent \textbf{Case 1.} $p \nmid B_{11}$.

The Smith normal form of $B_{11}$ gives $U,V\in \GL_{d_{2,1}'}(\Z_p)$ such that $UB_{11}V=\begin{pmatrix}1&O\\O&B_{11}^{rd}\end{pmatrix}$ for some $B_{11}^{rd}\in\M_{d_{2,1}'-1}(\Z_p)$. Then for every non-zero $x\in X_m$, we have
{\small
\begin{align*}
\cok_{\Z_p/p^2\Z_p}(P(x)) 
&\cong\cok_{\Z_p/p^2\Z_p}\left(\begin{pmatrix}U&O\\O&I_{d_{0,1}+d_{1,1}}\end{pmatrix}P(x)\begin{pmatrix}V&O\\O&I_{d_{0,1}+d_{1,1}}\end{pmatrix}\right) \\ 
&=\cok_{\Z_p/p^2\Z_p}\left(\begin{pmatrix}\begin{matrix}x&O\\O&xB_{11}^{rd}\end{matrix}&xUB_{12}&xUB_{13}\\xB_{21}V&pI_{d_{1,1}}+xB_{22}&xB_{23}\\xB_{31}V&xB_{32}&I_{d_{0,1}}+xB_{33}\end{pmatrix}+px^2C_1\right) \\
&=\cok_{\Z_p/p^2\Z_p}\left(\begin{pmatrix}x&O&xB_{12}^{u}&xB_{13}^{u}\\O&xB_{11}^{rd}&xB_{12}^{d}&xB_{13}^{d}\\xB_{21}^{l}&xB_{21}^{r}&pI_{d_{1,1}}+xB_{22}&xB_{23}\\xB_{31}^{l}&xB_{31}^{r}&xB_{32}&I_{d_{0,1}}+xB_{33}\end{pmatrix}+px^2C_1 \right) \\
&\cong\cok_{\Z_p/p^2\Z_p}\left(\begin{pmatrix}x&O&O&O\\O&xB_{11}^{rd}&xB_{12}^{d}&xB_{13}^{d}\\O&xB_{21}^{r}&pI_{d_{1,1}}+xB_{22}'&xB_{23}'\\O&xB_{31}^{r}&xB_{32}'&I_{d_{0,1}}+xB_{33}'\end{pmatrix}+px^2 \begin{pmatrix}
c & C_{ru}\\ 
C_{ld} & C_{rd}
\end{pmatrix} \right) \\ 
&\cong\cok_{\Z_p/p^2\Z_p} (Q(x))
\end{align*} }%
for 
$$
Q(t) := \begin{pmatrix}O&O&O\\O&pI_{d_{1,1}}&O\\O&O&I_{d_{0,1}} \end{pmatrix}
+t\begin{pmatrix}B_{11}^{rd}&B_{12}^{d}&B_{13}^{d}\\B_{21}^{r}&B_{22}'&B_{23}'\\B_{31}^{r}&B_{32}'&B_{33}'\end{pmatrix}+pt^2C_{rd}\in \M_{n-1}(\Z_p)[t].
$$
Here the last isomorphism is due to the relation
{\small
\begin{align*}
\cok_{\Z_p/p^2\Z_p} \left(\begin{pmatrix}x+pcx^2 &pf_1(x) \\pf_2(x)&  Q(x)\end{pmatrix}\right) 
& \cong \cok_{\Z_p/p^2\Z_p} \left(\begin{pmatrix}x+pcx^2&O\\O&Q(x) - pf_2(x)(x+pcx^2)^{-1}pf_1(x)\end{pmatrix}\right) \\
& \cong \cok_{\Z_p/p^2\Z_p} (Q(x))
\end{align*} }%
for every $x \not\equiv 0 \,\, (\bmod \; p)$.
Since $Q(t)$ is a first integral with $d_{0,Q(x_1)}=d_{0,P(x_1)}$, $d_{0,Q(x_i)}=d_{0,P(x_i)}-1$ for every $2 \le i \le m$ and $d_{1,Q(x_i)}=d_{1,P(x_i)}$ for every $1\le i\le m$, it satisfies the condition (\ref{item3d13}) for $r=1$.

\phantom{ }

\noindent \textbf{Case 2.} $p \mid B_{11}$ and [$p\nmid B_{12}$ or $p\nmid B_{21}$].

We may assume that $p \nmid B_{12}$. Choose invertible matrices $U$ and $V$ such that $UB_{12}V=\begin{pmatrix}1&O\\O&B_{12}^{rd}\end{pmatrix}$ and write $B_{11} = pB_{11}'$. Then for every non-zero $x\in X_m$, we have
{\small
\begin{align*}
\cok_{\Z_p/p^2\Z_p}(P(x))&\cong\cok_{\Z_p/p^2\Z_p}\left(\begin{pmatrix}U&O&O\\O&V^{-1}&O\\O&O&I_{d_{0,1}}\end{pmatrix}P(x)\begin{pmatrix}I_{d_{2,1}'}&O&O\\O&V&O\\O&O&I_{d_{0,1}}\end{pmatrix}\right)\\
&=\cok_{\Z_p/p^2\Z_p}\left(\begin{pmatrix}pxUB_{11}'&\begin{matrix}x&O\\O&xB_{12}^{rd}\end{matrix}&xUB_{13}\\xV^{-1}B_{21}&pI_{d_{1,1}}+xV^{-1}B_{22}V&xV^{-1}B_{23}\\xB_{31}&xB_{32}V&I_{d_{0,1}}+xB_{33}\end{pmatrix}+px^2C_1\right)\\
&=\cok_{\Z_p/p^2\Z_p}\left(\begin{pmatrix}pxB_{11}^{u}&x&O&xB_{13}^{u}\\
pxB_{11}^{d}&O&xB_{12}^{rd}&xB_{13}^{d}\\
xB^u_{21}&p+xB_{22}^{lu}&xB_{22}^{ru}&xB^u_{23}\\
xB^d_{21}&xB_{22}^{ld}&pI_{d_{1,1}-1}+xB_{22}^{rd}&xB^d_{23}\\
xB_{31}&xB_{32}^{l}&xB_{32}^{r}&I_{d_{0,1}}+xB_{33}
\end{pmatrix}+px^2C_1\right)\\
&\cong\cok_{\Z_p/p^2\Z_p}\left(\begin{pmatrix}O&x&O&O\\pxB_{11}^d&O&xB_{12}^{rd}&xB_{13}^{d}\\xB^{u'}_{21}&p&xB^{ru}_{22}&-pB_{13}^{u}+xB^{u'}_{23}\\xB^{d'}_{21}&O&pI_{d_{1,1}-1}+xB^{rd}_{22}&xB^{d'}_{23}\\xB_{31}'&O&xB_{32}^{r}&I_{d_{0,1}}+xB_{33}'\end{pmatrix}+px^2C_2\right)\\
&\cong\cok_{\Z_p/p^2\Z_p}\left(\begin{pmatrix}O&x+pcx^2&O\\xB_l+px^2C_l&O&\begin{pmatrix}O&O\\O&-pB_{13}^{u}\\pI_{d_{1,1}-1}&O\\O&I_{d_{0,1}}\end{pmatrix}+xB_r+px^2C_r\end{pmatrix}\right).
\end{align*} }%
Then 
$$
Q(t):=\begin{pmatrix}O&O&O\\O&O&-pB_{13}^{u}\\O&pI_{d_{1,1}-1}&O\\O&O&I_{d_{0,1}}\end{pmatrix}+t\begin{pmatrix}B_l&B_r\end{pmatrix}+pt^2\begin{pmatrix}C_l&C_r\end{pmatrix} \in \M_{n-1}(\Z_p)[t]
$$ 
satisfies the condition (\ref{item3d12}) for $r=1$.

\phantom{ }

\noindent \textbf{Case 3.} [$p \mid B_{11}, B_{12}, B_{21}$] and [$p\nmid B_{13}$ or $p\nmid B_{31}$].

We may assume that $p \nmid B_{13}$.  
Choose invertible matrices $U$ and $V$ such that $UB_{13}V=\begin{pmatrix}1&O\\O&B_{13}^{rd}\end{pmatrix}$ and write $B_{ij} = pB'_{ij}$ for $(i,j) \in \left \{ (1,1),(1,2) \right \}$.
Then for every non-zero $x\in X_m$, we have
{\small 
\begin{align*}
\cok_{\Z_p/p^2\Z_p}(P(x))&\cong\cok_{\Z_p/p^2\Z_p}\left(\begin{pmatrix}U&O&O\\O&I_{d_{1,1}}&O\\O&O&V^{-1}\end{pmatrix}P(x)\begin{pmatrix}I_{d_{2,1}'}&O&O\\O&I_{d_{1,1}}&O\\O&O&V\end{pmatrix}\right)\\
&=\cok_{\Z_p/p^2\Z_p}\left(\begin{pmatrix}pxUB'_{11}&pxUB'_{12}&\begin{matrix}x&O\\O&xB_{13}^{rd}\end{matrix}\\
xB_{21}&pI_{d_{1,1}}+xB_{22}&xB_{23}V\\
xV^{-1}B_{31}&xV^{-1}B_{32}&I_{d_{0,1}}+xV^{-1}B_{33}V
\end{pmatrix}+px^2C_1\right)\\
&=\cok_{\Z_p/p^2\Z_p}\left(\begin{pmatrix}pxB_{11}^{u}&pxB_{12}^{u}&x&O\\
pxB_{11}^{d}&pxB_{12}^{d}&O&xB_{13}^{rd}
\\xB_{21}&pI_{d_{1,1}}+xB_{22}&xB_{23}^{l}&xB_{23}^{r}\\
xB^u_{31}&xB^u_{32}&1+xB_{33}^{lu}&xB_{33}^{ru}\\
xB_{31}^d&xB_{32}^d&xB_{33}^{ld}&I_{d_{0,1}-1}+xB_{33}^{rd}\end{pmatrix}+px^2C_1\right)\\
&\cong\cok_{\Z_p/p^2\Z_p}\left(\begin{pmatrix}O&O&x&O\\
pxB_{11}^{d}&pxB_{12}^{d}&O&xB_{13}^{rd}\\
xB_{21}'&pI_{d_{1,1}}+xB_{22}'&O&xB_{23}^{r}\\
-pB_{11}^{u}+xB^{u'}_{31}&-pB_{12}^{u}+xB^{u'}_{32}&1&xB_{33}^{ru}\\
xB_{31}^{d'}&xB_{32}^{d'}&O&I_{d_{0,1}-1}+xB_{33}^{rd}
\end{pmatrix}+px^2C_2\right)\\
&\cong\cok_{\Z_p/p^2\Z_p}\left(\begin{pmatrix}O&x+pcx^2&O \\ 
\begin{pmatrix}O&O\\O&pI_{d_{1,1}}\\-pB_{11}^{u}&-pB_{12}^{u}\\O&O\end{pmatrix} +xB_l+px^2C_l&O&\begin{pmatrix}O\\O\\O\\I_{d_{0,1}-1}\end{pmatrix}
+xB_r+px^2C_r\end{pmatrix}\right).
\end{align*} }%
Then 
$$
Q(t):=\begin{pmatrix} O&O&O\\O&pI_{d_{1,1}}&O\\-pB_{11}^{u}&-pB_{12}^{u}&O\\O&O&I_{d_{0,1}-1} \end{pmatrix}+t\begin{pmatrix}B_l&B_r\end{pmatrix}+pt^2\begin{pmatrix}C_l&C_r\end{pmatrix} \in \M_{n-1}(\Z_p)[t]
$$ 
satisfies the condition (\ref{item3d11}) for any $r$. 

\phantom{ }

\noindent \textbf{Case 4.} [$p \mid B_{11}, B_{12}, B_{21}, B_{13}, B_{31}$] and $p\nmid B_{22}$.

Choose invertible matrices $U$ and $V$ such that $UB_{22}V=\begin{pmatrix}1&O\\O&B_{22}^{rd} \end{pmatrix}$ and write $B_{ij} = pB'_{ij}$ for $(i,j) \in \left \{ (1,1),(1,2), (2,1) \right \}$. Let $UV=\begin{pmatrix}d&D_1 \\ D_2&D_3\end{pmatrix} \in \GL_{d_{1,1}}(\Z_p)$. Then for every non-zero $x\in X_m$, we have
{\small
\begin{align*}
\cok_{\Z_p/p^2\Z_p}(P(x))
&\cong\cok_{\Z_p/p^2\Z_p}\left(\begin{pmatrix}I_{d_{2,1}'}&O&O\\O&U&O\\O&O&I_{d_{0,1}}\end{pmatrix}P(x)\begin{pmatrix}I_{d_{2,1}'}&O&O\\O&V&O\\O&O&I_{d_{0,1}}\end{pmatrix}\right)\\
&=\cok_{\Z_p/p^2\Z_p}\left(\begin{pmatrix}pxB'_{11}&pxB'_{12}V&xB_{13}\\pxUB'_{21}&\begin{matrix}x+pd&pD_1\\pD_2&pD_3+xB_{22}^{rd}\end{matrix}&xUB_{23}\\xB_{31}&xB_{32}V&I_{d_{0,1}}+xB_{33}\end{pmatrix}+px^2C_1\right)\\
&=\cok_{\Z_p/p^2\Z_p}\left(\begin{pmatrix}pxB'_{11}&pxB_{12}^{l}&pxB_{12}^{r}&xB_{13}\\
pxB_{21}^{u}&x+pd&pD_1&xB_{23}^{u}\\
pxB_{21}^{d}&pD_2&pD_3+xB_{22}^{rd}&xB_{23}^{d}\\
xB_{31}&xB_{32}^{l}&xB_{32}^{r}&I_{d_{0,1}}+xB_{33}\end{pmatrix}+px^2C_1\right)\\
&\cong\cok_{\Z_p/p^2\Z_p}\left(\begin{pmatrix}
pxB'_{11}&O&pxB_{12}^{r}&xB_{13}'\\
O&x+pd&pD_1&-pdB_{23}^{u}\\
pxB_{21}^{d}&pD_2&pD_3+xB_{22}^{rd}&-pD_2B_{23}^{u}+xB_{23}^{d}\\
xB_{31}'&-pdB_{32}^{l}&-pB_{32}^{l}D_1+xB_{32}^{r}&I_{d_{0,1}}-pdB_{32}^{l}B_{23}^{u}+xB_{33}'
\end{pmatrix}+px^2C_2\right)\\
&\cong\cok_{\Z_p/p^2\Z_p}\left(\begin{pmatrix}xB_{11}'+px^2C_{lu}&O&xB_{ru}+px^2C_{ru}\\O&x+p(d+cx^2)&O\\xB_{ld}+px^2C_{ld}&O&\begin{pmatrix}pD_3&pA_{ru}\\pA_{ld}&I_{d_{0,1}}+pA_{rd}\end{pmatrix} +xB_{rd}+px^2C_{rd}\end{pmatrix}\right).
\end{align*} }%
Then
$$
Q(t) :=\begin{pmatrix} 
O&O&O \\ O&pD_3&pA_{ru} \\ O&pA_{ld}&I_{d_{0,1}}+pA_{rd}
\end{pmatrix}
+t\begin{pmatrix}
B_{11}'&B_{ru} \\ B_{ld}&B_{rd}
\end{pmatrix}
+pt^2\begin{pmatrix} 
C_{lu}&C_{ru} \\ C_{ld}&C_{rd}
\end{pmatrix} \in \M_{n-1}(\Z_p)[t]
$$ 
satisfies the condition (\ref{item3d12}) for $r=1$ if $\rank_{\F_p}(\overline{D_3}) \in \left \{ d_{1,1}-1, d_{1,1}-2 \right \}$. This is true by the inequality
$$
\rank_{\F_p}(\overline{D_3}) \ge \rank_{\F_p}\begin{pmatrix}
\overline{D_2} & \overline{D_3}
\end{pmatrix} - 1 \ge \rank_{\F_p} (\overline{UV}) - 2 = d_{1,1}-2.
$$

\phantom{ }

\noindent \textbf{Case 5.} [$p \mid B_{11}, B_{12}, B_{21}, B_{13}, B_{31}, B_{22}$] and [$p \nmid B_{23}$ or $p\nmid B_{32}$]. 

We may assume that $p\nmid B_{23}$. Choose invertible matrices $U$ and $V$ such that $UB_{23}V=\begin{pmatrix}1&O\\O&B_{23}^{rd}\end{pmatrix}$ and write $B_{ij} = pB'_{ij}$ for each $i+j \le 4$. Then for every non-zero $x\in X_m$, we have
{\small
\begin{align*}
\cok_{\Z_p/p^2\Z_p}(P(x))&\cong\cok_{\Z_p/p^2\Z_p}\left(\begin{pmatrix}I_{d_{2,1}'}&O&O\\O&U&O\\O&O&V^{-1}\end{pmatrix}P(x)\begin{pmatrix}I_{d_{2,1}'}&O&O\\O&U^{-1}&O\\O&O&V\end{pmatrix}\right)\\
&=\cok_{\Z_p/p^2\Z_p}\left(\begin{pmatrix}pxB'_{11}&pxB'_{12}U^{-1}&pxB'_{13}V\\pxUB'_{21}&pI_{d_{1,1}}+pxUB_{22}'U^{-1}&\begin{matrix}x&O\\O&xB_{23}^{rd}\end{matrix}\\pxV^{-1}B'_{31}&xV^{-1}B_{32}U^{-1}&I_{d_{0,1}}+xV^{-1}B_{33}V\end{pmatrix}+px^2C_1\right)\\
&=\cok_{\Z_p/p^2\Z_p}
\left(\begin{pmatrix}pxB_{11}'&pxB^l_{12}&pxB^r_{12}&pxB_{13}^{l}&pxB_{13}^{r}\\
pxB_{21}^{u}&p+pxb_{22}^{lu}&pxB_{22}^{ru}&x&O\\
pxB_{21}^{d}&pxB_{22}^{ld}&pI_{d_{1,1}-1}+pxB_{22}^{rd}&O&xB_{23}^{rd}\\
pxB_{31}^u&xb_{32}^{lu}&xB_{32}^{ru}&1+xb_{33}^{lu}&xB_{33}^{ru}\\
pxB_{31}^d&xB_{32}^{ld}&xB_{32}^{rd}&xB_{33}^{ld}&I_{d_{0,1}-1}+xB_{33}^{rd}
\end{pmatrix}+px^2C_1\right)\\
&\cong\cok_{\Z_p/p^2\Z_p}
\left(\begin{pmatrix}
pxB_{11}'&pxB^l_{12}&pxB^r_{12}&O&pxB_{13}^{r}\\
O&p&O&x&O\\
pxB_{21}^{d}&pxB_{22}^{ld}&pI_{d_{1,1}-1}+pxB_{22}^{rd}&O&xB_{23}^{rd}\\
pA_1+pxB_{31}^{u'}&pa_2+xb_{32}^{lu'}&pA_3+xB_{32}^{ru'}&1&xB_{33}^{ru}\\
pxB_{31}^{d'}&pA_4+xB_{32}^{ld'}&xB_{32}^{rd'}&O&I_{d_{0,1}-1}+xB_{33}^{rd}
\end{pmatrix}+px^2C_2\right)\\
&\cong\cok_{\Z_p/p^2\Z_p}
\left(\begin{pmatrix}
*&*&*&O&*\\
O&O&O&x+pcx^2&O\\
*&*&pI_{d_{1,1}-1}+*&O&*\\
pA_1+*&p(x+pcx^2)^{-1}+pa_2+*&pA_3+*&O&*\\
*&pA_4+*&*&O&I_{d_{0,1}-1}+* 
\end{pmatrix} \right)\\
&\cong\cok_{\Z_p/p^2\Z_p}(px^{-1}J+P_1(x)),
\end{align*} }%
where each * is of the form $xB+px^2C$, {\small $J :=\begin{pmatrix}O&O&O&O\\O&O&O&O\\O&-1&O&O\\O&O&O&O\end{pmatrix} \in \M_{n-1}(\Z_p)$ }%
($J_{n-d_{0,1}, d_{2,1}'+1}=-1$) and 
{\small $$
P_1(t) := \begin{pmatrix}O&O&O&O\\O&O&pI_{d_{1,1}-1}&O\\pA_1&pa_2&pA_3&O\\O&pA_4&O&I_{d_{0,1}-1}\end{pmatrix} + tB'+pt^2C' \in \M_{n-1}(\Z_p)[t].
$$ }%
Now we have $d_{0,P_1(x_i)}=d_{0,i}-1$ and $\lvert d_{1,P_1(x_i)}-d_{1,i}\rvert\le1$ by Lemma \ref{lem3d}. Precisely, 
$$
d_{1, P_1(x_1)} = \begin{cases}
d_{1,1}-1 &\text{if } p \mid A_1 \text{ and } p \mid a_2, \\ 
d_{1,1} &\text{otherwise}
\end{cases}
$$
and for every $2 \le i \le m$, 

$$
d_{1,P_1(x_i)} = \begin{cases}
d_{1,i}+1 &\text{if }J\text{ is singular of order }1\text{ at }x_i^{-1}\text{ over }P_1(x_i),\\ 
d_{1,i}-1 &\text{if }J\text{ is singular of order }1\text{ at }0\text{ over }P_1(x_i),\\ 
d_{1,i} &\text{otherwise.}
\end{cases}
$$
Let $P_1(t)=\begin{pmatrix}q_1(t) \; \cdots \; q_{n-1}(t)\end{pmatrix}$ and $q_i(t) = \mathbf{a}_i +t\mathbf{b}_i +pt^2\mathbf{c}_i$ for some $\mathbf{a}_i, \mathbf{b}_i, \mathbf{c}_i \in \Z_p^{n-1}$. 

\phantom{ }

\noindent \textbf{Case 5.1.} $J$ is singular of order $1$ at $x_r^{-1}$ over $P_1(x_r)$ for some $2\le r\le m$.

Let $e_i$ be the $i$-th unit vector in $\Z_p^{n-1}$, so that the non-zero column of $J$ is $-e_{n-d_{0,1}}$. The proof of Lemma \ref{lem3d} tells us that the $(d_{2,1}'+1)$-th column of $P_1(x_r)$ is a $\Z_p$-linear combination of other columns of $P_1(x_r)$, i.e.
$$
q_{d_{2,1}'+1}(x_r) - px_r^{-1}e_{n-d_{0,1}} \equiv \sum_{i\neq d_{2,1}'+1}c_iq_i(x_r) \,\, (\bmod \; p^2)
$$
for some $c_i \in \Z_p$ ($i \neq d_{2,1}'$). Then,
{\small
\begin{align*}
(pt^{-1}J+P_1(t))\begin{pmatrix}I_{d_{2,1}'}&\begin{matrix}-c_1 \\ \vdots \\ -c_{d_{2,1}'}\end{matrix}&O\\O&1&O\\O&\begin{matrix}-c_{d_{2,1}'+2} \\ \vdots\\-c_{n-1}\end{matrix}&I_{d_{0,1}+d_{1,1}-2}\end{pmatrix}
=\begin{pmatrix}q_1(t) \; \cdots \; q_{d_{2,1}'}(t)&g(t)&q_{d_{2,1}'+2}(t) \; \cdots \; q_{n-1}(t)\end{pmatrix}
\end{align*} }%
and $g(x_r) \equiv 0 \,\, (\bmod \; p^2)$ so we have $g(t) \equiv (t-x_r)f(t) \,\, (\bmod \; p^2)$ for
$$ 
f(t) := pt^{-1}x_r^{-1}e_{n-d_{0,1}} + \mathbf{a}'+pt \mathbf{b}' 
$$
($\mathbf{a}', \mathbf{b}' \in \Z_p^{n-1}$). Now 
$$
P_2(t) := \begin{pmatrix}q_1(t) \; \cdots \;  q_{d_{2,1}'}(t)&tf(t)&q_{d_{2,1}'+2}(t) \; \cdots \; q_{n-1}(t)\end{pmatrix} \in \M_{n-1}(\Z_p)[t]
$$
satisfies the condition (\ref{item3d11}) or (\ref{item3d13}) for the same $r$.
\begin{itemize}
    \item For $i \not\in \left \{ 1,r \right \}$, we have $\cok_{\Z_p/p^2\Z_p}(P(x_i))\cong\cok_{\Z_p/p^2\Z_p}(px_i^{-1}J+P_1(x_i))\cong\cok_{\Z_p/p^2\Z_p}(P_2(x_i))$ so $d_{0, P_2(x_i)} = d_{0,i}-1$ and $d_{1, P_2(x_i)} = d_{1,i}$.

    \item For $i=1$, we have {\small $P_2(0)=\begin{pmatrix}O&O&O&O\\O&O&pI_{d_{1,1}-1}&O\\pA_1&px_r^{-1}&pA_3&O\\O&O&O&I_{d_{0,1}-1}\end{pmatrix}$} so $d_{0, P_2(x_1)} = d_{0,1}-1$ and $d_{1, P_2(x_1)} = d_{1,1}$.

    \item For $i=r$, we have
    \begin{equation*}
    \begin{split}
    \cok_{\Z_p/p^2\Z_p}(P(x_r)) & \cong\cok_{\Z_p/p^2\Z_p}(px_r^{-1}J+P_1(x_r)) \\
    & \cong\cok_{\Z_p/p^2\Z_p}\begin{pmatrix}q_1(t) \; \cdots \; q_{d_{2,1}'}(t)&q_{d_{2,1}'+2}(t) \; \cdots \; q_{n-1}(t)\end{pmatrix}    
    \end{split}    
    \end{equation*}
    so $(d_{0,P_2(x_r)}, d_{1,P_2(x_r)}) \in \left \{ (d_{0,r}, d_{1,r}), (d_{0,r}-1, d_{1,r}), (d_{0,r}-1, d_{1,r}+1) \right \}$.
\end{itemize}

\phantom{ }

\noindent \textbf{Case 5.2.} $J$ is not singular of order $1$ at $x_r^{-1}$ over $P_1(x_r)$ for all $2\le r\le m$.

Let $P_b(t) := p(a_2+1+bt)J + P_1(t) \in \M_{n-1}(\Z_p)[t]$ for $b \in \left \{ 0, 1, \cdots, p-1 \right \}$. Then we have
$$
\cok_{\Z_p/p^2\Z_p}(px_r^{-1}J+P_1(x_r))\cong\cok_{\Z_p/p^2\Z_p}(P_{b}(x_r))
$$
unless $J$ is singular of order $1$ at $a_2+1+bx_r$ over $P_1(x_r)$. Since $p>m-1$, one can choose $b_0 \in \left \{ 0, 1, \cdots, p-1 \right \}$ such that $J$ is not singular of order $1$ at $a_2+1+b_0x_r$ over $P_1(x_r)$ for every $2 \le r \le m$. Now $P_{b_0}(t)$ is a first integral and satisfies the condition (\ref{item3d11}) for $r=1$.

\begin{itemize}
    \item For $i \neq 1$, we have $\cok_{\Z_p/p^2\Z_p}(P(x_i))\cong\cok_{\Z_p/p^2\Z_p}(px_i^{-1}J+P_1(x_i))\cong\cok_{\Z_p/p^2\Z_p}(P_{b_0}(x_i))$ so $d_{0, P_{b_0}(x_i)} = d_{0,i}-1$ and $d_{1, P_{b_0}(x_i)} = d_{1,i}$.

    \item For $i=1$, {\small $P_{b_0}(0)=\begin{pmatrix}O&O&O&O\\O&O&pI_{d_{1,1}-1}&O\\pA_1&-p&pA_3&O\\O&pA_4&O&I_{d_{0,1}-1}\end{pmatrix}$} so $d_{0, P_{b_0}(x_1)} = d_{0,1}-1$ and $d_{1, P_{b_0}(x_1)} = d_{1,1}$.
\end{itemize}

\phantom{ }

In the remaining cases, we have $B_{ij} \equiv 0 \,\, (\bmod \; p)$ except for $(i,j)=(3,3)$ so that $d_{0,1}=\underset{1\le i\le m}{\max}d_{0,i}$. By the same reason, we have $d_{0,1} = \cdots = d_{0, m}$. Then the matrix $I_{d_{0,1}}+x_iB_{33}+px_i^2C_{33}$ is invertible for each $i$ and thus $Q(t):=\begin{pmatrix}O&O\\O&pI_{d_{1,1}}\end{pmatrix}+t\begin{pmatrix}B_{11}&B_{12}\\B_{21}&B_{22}\end{pmatrix}+pt^2\begin{pmatrix}C_{11}&C_{12}\\C_{21}&C_{22}\end{pmatrix}$ satisfies the condition (\ref{item3d2}).
\end{proof}

 \subsection{Necessary condition for an element of \texorpdfstring{$\calC_{X_4}$}{calCX4}} \label{Sub32}

For a given $(n;H_1,\cdots,H_m)\in\calC_{X_m,1,\infty}$, there exists a first integral $P(t)\in\M_n(\Z_p)[t]$ such that $\cok(P(x_i))\cong H_i$ for each $i$. We will use Lemma \ref{lem3e} repeatedly until we obtain $(n';H'_1,\cdots,H'_m)\in\calC_{X_m,1,\infty}$ which satisfies $\rank_{\F_p}(H'_i/pH'_i)=n'$ for each $i$. For the case $m=4$, the elements of $(n';H'_1,\cdots,H'_4) \in \calC_{X_4,1,2}$ which satisfies $\rank_{\F_p}(H'_i/pH'_i)=n'$ for each $i$ bijectively corresponds to the elements of $\calC_{X_4,2,1}$ by Proposition \ref{prop2c}(\ref{item2c4}). The following lemma, which provides an information about the set $\calC_{X_4,2,1}$, is a generalization of Lemma \ref{lem3c}. 

\begin{lemma} \label{lem3f}
If $(n;H_1,\cdots,H_m)\in\calC_{X_m,l,1}$ and $l \le m$, then
$$
\sum_{i=1}^{m}d_{0, i} - (m-l)\alpha_0\ge0
$$
for some non-negative integer $\alpha_0\ge\underset{1\le i\le m}{\max}d_{0,i}$.
\end{lemma}

\begin{proof}
We use induction on $n$. The case $n=1$ follows from Example \ref{ex2d}. Now assume that $n>1$ and the lemma holds for every $n'<n$. Suppose that there exists $(n;H_1,\cdots,H_m)\in\calC_{X_m,l,1}$ such that 
$\sum_{i=1}^{m}d_{0,i} < (m-l)\underset{1\le i\le m}{\max} d_{0,i}$. 
As in the proof of Lemma \ref{lem3c}, we may assume that $d_{0,1}=\underset{1\le i\le m}{\max}d_{0,i}$ and there exists an $l$-th integral
$$
P(t)=\begin{pmatrix}O&O\\O&I_{d_{0,1}}\end{pmatrix}+tA_1+\cdots+t^lA_l \in \M_n(\Z_p)[t]
$$ 
such that $\cok_{\Z_p/p\Z_p}(P(x_i))\cong H_i/pH_i$ for each $i$.

Assume that $d_{0,1} < n$ and let $Q(t)\in\M_{n-1}(\Z_p)[t]$ be the $l$-th integral which is obtained by eliminating first column and row from $P(t)$. Then we have $d_{0,Q(x_1)}=d_{0,1}$ and $d_{0,Q(x_i)}\le d_{0,i} \le d_{0,1}$ for $2\le i\le m$. The induction hypothesis implies that
$$\sum_{i=1}^{m} d_{0,i}-(m-l)\underset{1\le i\le m}{\max} d_{0,i}
\ge \sum_{i=1}^{m} d_{0,Q(x_i)}-(m-l)\underset{1\le i\le m}{\max} d_{0,Q(x_i)} \ge 0,
$$
which is a contradiction. Thus we have $d_{0,1}=n$. Now consider the first integral
$$
P_1(t) := \begin{pmatrix}tI&&&&A_0
\\-I&tI&&&A_1\\&\ddots&\ddots&&\vdots
\\&&-I&tI&A_{l-2}
\\&&&-I&A_{l-1}+tA_l
\end{pmatrix}\in\M_{ln}(\Z_p)[t],
$$
where $A_0=P(0)=\begin{pmatrix}O&O\\O&I_{d_{0,1}}\end{pmatrix}$. It satisfies $\cok(P_1(t))\cong\cok(P(t))$ so that $d_{0,P_1(x_i)}=d_{0,i}+(l-1)n$ for every $1\le i\le m$ and $\underset{1\le i\le m}{\max} d_{0,P_1(x_i)} = ln$. Then we have
$$
\sum_{i=1}^{m}d_{0,i}-(m-l)n=\sum_{i=1}^{m}d_{0,P_1(x_i)}-(m-1)ln\ge0
$$
by Lemma \ref{lem3c}, which is a contradiction. This finishes the proof.
\end{proof}

Now we prove a necessary condition for an element of $\calC_{X_m,1,2}$ using Lemma \ref{lem3e} and \ref{lem3f}. 


\begin{theorem} \label{thm3g}
If $(n;H_1,\cdots,H_m)\in\calC_{X_m,1,2}$, then
\begin{gather*}
\sum_{i=1}^md_{0,i}-(m-1)\alpha_0\ge 0, \\
\left(\sum_{i=1}^md_{0,i}-(m-1)\alpha_0\right)+\left(\sum_{i=1}^m\min(d_{1,i},\alpha_1)-(m-2)\alpha_1\right)\ge 0
\end{gather*}
for some non-negative integers $\alpha_0$ and $\alpha_1$ such that $\alpha_0\ge\underset{1\le i\le m}{\max}d_{0,i}$ and $2 \alpha_0+\alpha_1\ge\underset{1\le i\le m}{\max}(2d_{0,i}+d_{1,i})$.
\end{theorem}

\begin{proof}
We use induction on $n$. The case $n=1$ follows from Example \ref{ex2d}. Now assume that $n>1$ and the theorem holds for every $n'<n$. Let $P(t)\in\M_n(\Z_p)[t]$ be a first integral such that $\cok_{\Z_p/p^2\Z_p}(P(x_i)) \cong H_i/p^2H_i$ for each $i$. If $d_{0,i}=0$ for each $i$, then we have $(n; pH_1, \cdots, pH_m) \in \calC_{X_m, 2,1}$ by Proposition \ref{prop2c}(\ref{item2c4}) and $d_{0, pH_i} = d_{1,i}$. By Lemma \ref{lem3f}, there exists $\beta_0 \ge \underset{1\le i\le m}{\max} d_{0,pH_i}$ such that $\sum_{i=1}^{m} d_{0,pH_i} - (m-2) \beta_0 \ge 0$. Now $(\alpha_0, \alpha_1) = (0, \beta_0)$ satisfies the desired properties. Otherwise, choose a first integral $Q(t) \in \M_{n'}(\Z_p)$ with $n' < n$ satisfying one of the conditions in Lemma \ref{lem3e}. By the induction hypothesis, we have 
\begin{gather*}
\sum_{i=1}^md_{0,Q(x_i)}-(m-1)\alpha'_0\ge 0, \\
\left(\sum_{i=1}^m d_{0,Q(x_i)}-(m-1)\alpha'_0\right)+\left(\sum_{i=1}^m\min(d_{1,Q(x_i)},\alpha'_1)-(m-2)\alpha'_1\right)\ge 0
\end{gather*}
for some $\alpha'_0,\alpha_1'\in\Z_{\ge0}$ such that $\alpha'_0 \ge \underset{1\le i\le m}{\max}d_{0,Q(x_i)}$ and $2\alpha'_0+\alpha'_1 \ge \underset{1\le i\le m}{\max}(2d_{0,Q(x_i)}+d_{1,Q(x_i)})$. Now we divide the cases according to the condition of $Q(t)$.

\phantom{ }

\noindent \textbf{Case 1.} $Q(t)$ satisfies the condition (\ref{item3d1}) of Lemma \ref{lem3e}: Let $(\alpha_0, \alpha_1) = (\alpha_0'+1, \alpha_1')$. Then we have
\begin{equation*}
\begin{split}
\alpha_0 & \ge \underset{1\le i\le m}{\max}d_{0,Q(x_i)} + 1 = \underset{1\le i\le m}{\max}d_{0,i}, \\
2\alpha_0+\alpha_1 & \ge \underset{1\le i\le m}{\max}(2(d_{0,Q(x_i)}+1)+d_{1,Q(x_i)})\ge\underset{1\le i\le m}{\max}(2d_{0,i}+d_{1,i}), \\
\sum_{i=1}^md_{0,i}-(m-1)\alpha_0 & \ge \sum_{i=1}^md_{0,Q(x_i)}-(m-1)\alpha'_0 \ge 0.
\end{split}   
\end{equation*}
If $Q(t)$ satisfies the condition (\ref{item3d11}), then we have
\begin{equation*}
\begin{split}
0 & \le \left(\sum_{i=1}^{m}d_{0,Q(x_i)}-(m-1)\alpha'_0+1\right)+\left(\sum_{i=1}^{m}\min(d_{1,Q(x_i)},\alpha'_1)-(m-2)\alpha'_1-1\right) \\
& \le \left(\sum_{i=1}^{m}d_{0,i}-(m-1)\alpha_0\right)+ \left(\sum_{i=1}^{m}\min(d_{1,i},\alpha_1)-(m-2)\alpha_1\right).
\end{split}    
\end{equation*}
If $Q(t)$ satisfies the condition (\ref{item3d12}) or (\ref{item3d13}), then we have
\begin{equation*}
\begin{split}
0 & \le \left(\sum_{i=1}^{m}d_{0,Q(x_i)}-(m-1)\alpha'_0\right)+\left(\sum_{i=1}^{m}\min(d_{1,Q(x_i)},\alpha'_1)-(m-2)\alpha'_1\right) \\
& \le \left(\sum_{i=1}^{m}d_{0,i}-(m-1)\alpha_0\right)+ \left(\sum_{i=1}^{m}\min(d_{1,i},\alpha_1)-(m-2)\alpha_1\right).
\end{split}    
\end{equation*}

\phantom{ }

\noindent \textbf{Case 2.} $Q(t)$ satisfies the condition (\ref{item3d2}) of Lemma \ref{lem3e}: In this case, $\alpha'_0=0$ and $\sum_{i=1}^{m} \min(d_{1, i}, \alpha'_1) - (m-2)\alpha'_1 \ge 0$ for some $\alpha'_1 \ge \underset{1\le i\le m}{\max}d_{1,i}$. Let $d_0 = d_{0,i}$ for each $i$ and $(\alpha_0, \alpha_1) = (d_0, \alpha_1')$. Then we have
\begin{equation*}
\begin{split}
\alpha_0 & \ge \underset{1\le i\le m}{\max}d_{0,i}, \\
2\alpha_0+\alpha_1 & \ge 2\underset{1\le i\le m}{\max} d_{0,i} + \underset{1\le i\le m}{\max} d_{1,i} \ge\underset{1\le i\le m}{\max}(2d_{0,i}+d_{1,i}), \\
\sum_{i=1}^md_{0,i}-(m-1)\alpha_0 & = md_0 - (m-1)d_0 \ge 0, \\
d_0 & \le \left(\sum_{i=1}^{m}d_{0,i}-(m-1)\alpha_0\right)+ \left(\sum_{i=1}^{m}\min(d_{1,i},\alpha_1)-(m-2)\alpha_1\right). \qedhere
\end{split}   
\end{equation*}
\end{proof}

As a corollary of Theorem \ref{thm3g}, we prove one direction of Theorem \ref{thm1h} for the case $m=4$.

\begin{corollary} \label{cor3h}
The inclusion
$$
\calC_{X_4}\subset\left \{ (H_1, H_2, H_3, H_4) \in \calM_{\Z_p}^4 : \begin{matrix}
s_1=s_2=s_3=s_4, \; 3d_{1, i} \le D_1 (1 \le i \le 4) \text{ and} \\
d_{1, i} + 2(d_{1,j}+d_{2,j}) \le D_1+D_2 \;\; (1 \le i,j \le 4)
\end{matrix} \right \}
$$
holds where $s_i=\rank_{\F_p}(H_i/pH_i)$ ($1\le i\le 4$) and $D_r=\sum_{i=1}^{4} d_{r,i}$ for $r=1, 2$.
\end{corollary}

\begin{proof}
Suppose that $(H_1,H_2,H_3,H_4)\in\calC_{X_4}$. By Proposition \ref{prop2c}, we have
$$
s_1=s_2=s_3=s_4\text{ and }(n;pH_1,pH_2,pH_3,pH_4)\in\calC_{X_4,1,\infty}=\calC_{X_4,1,2}
$$ 
for some $n\in\Z_{\ge1}$. Let $P(t)\in\M_n(\Z_p)[t]$ be a first integral which satisfies $\cok_{\Z_p/p^2\Z_p}(P(x_i))\cong pH_i/p^3H_i$ so that $d_{0,P(x_i)}=d_{1,i}$ and $d_{1,P(x_i)}=d_{2,i}$ for each $i$. By Theorem \ref{thm3g}, there exist $\alpha_0,\alpha_1\in\Z_{\ge0}$ such that
$$
D_1-3\alpha_0\ge0\text{ and }\left(D_1-3\alpha_0\right)+\left(\sum_{i=1}^4\min(d_{2,i},\alpha_1)-2\alpha_1\right)\ge0
$$
with $\alpha_0\ge\underset{1\le i\le 4}{\max}d_{1,i}$ and $2\alpha_0+\alpha_1\ge\underset{1\le i\le 4}{\max}(2d_{1,i}+d_{2,i})$. 
\begin{itemize}
    \item For every $1\le i\le 4$, we have $3d_{1,i} \le 3\alpha_0 \le D_1$.

    \item For every $1\le i,j\le 4$, we have
    \begin{equation*}
    \begin{split}
    D_1+D_2-d_{2,j} & \ge \left(D_1-3\alpha_0\right)+\left(\sum_{i_0=1}^4\min(d_{2,i_0},\alpha_1)-2\alpha_1\right) + (3\alpha_0+\alpha_1) \\
    & \ge \alpha_0 + (2\alpha_0+\alpha_1) \\
    & \ge d_{1,i}+(2d_{1,j}+d_{2,j})    
    \end{split}    
    \end{equation*}
    so we conclude that $d_{1,i}+2(d_{1,j}+d_{2,j}) \le D_1+D_2$. \qedhere
\end{itemize}
\end{proof}

\subsection{Zone Theory} \label{Sub33}

In this section, we prove the sufficient condition of Theorem \ref{thm1h} for $m=4$. We will find a generating set of $\calC_{X_4}$ and will prove that each element in a generating set satisfies the desired condition. To do this, we introduce a way to visualize an element of $\calB_m$.

\begin{definition} \label{def3j}
For given $k\in\Z_{\ge0}$, a $k$-\textit{presentation} of $(n;H_1,\cdots,H_m)\in\calB_m$ (denoted by $\Prs_k(n;H_1,\cdots,H_m)$) is an $m \times n$ matrix with entries in $\{0,1,\cdots,k-1,k^+\}$ whose $i$-th row contains $d_{r, i}$ numbers of $r$'s ($0 \le r \le k-1$) and $\sum_{r\in\extz_{\ge k}}d_{r, i}$ numbers of $k^+$'s. (It is unique up to ordering of the numbers in each row.) Each entry in a presentation of a block is called a \textit{type}.    
\end{definition}

\begin{figure}[ht]
$$
\begin{bmatrix*}[l]
0&0&1&1&1&2&2&2&2 \\ 
1&0&0&1&1&1&3^+&3^+&0 \\
0&0&3^+&3^+&0&0&0&3^+&3^+
\end{bmatrix*}
$$ 
\caption{A $3$-presentation of $\begin{pmatrix}
9;(\Z_p/p\Z_p)^3\times(\Z_p/p^2\Z_p)^4,
(\Z_p/p\Z_p)^4\times(\Z_p/p^3\Z_p)^2,\Z_p^4
\end{pmatrix}\in\calB_4$}
\label{fig1}
\end{figure}

\begin{remark}\label{rmk3k}
For given $l\in\extz_{\ge0}$ and $k \in \Z_{\ge0}$, assume that $\Prs_k(n;H_1,\cdots,H_m)=\Prs_k(n;H'_1,\cdots,H'_m)$. Then we have $H_i/p^kH_i\cong H'_i/p^kH'_i$ for each $i$ so $(n;H_1,\cdots,H_m)\in\calC_{X_m,l,k}$ if and only if $(n;H'_1,\cdots,H'_m)\in\calC_{X_m,l,k}$. Hence it is enough to consider a $k$-presentation of $(n;H_1,\cdots,H_m)$ to determine whether it is an element of $\calC_{X_m,l,k}$ or not.
\end{remark}

A $k$-presentation of a sum of two elements in $\calB_{m}$ is given by the concatenation of a $k$-presentation of each element. By Proposition \ref{prop2c}(\ref{item2c1}), the set of $k$-presentations of elements in $\calC_{X_m,l,k}$ is closed under concatenation.

\begin{example} \label{ex3l}
The equation $(3;(\Z_p/p\Z_p)^2,1)+(2;1,\Z_p/p\Z_p)=(5;(\Z_p/p\Z_p)^2,\Z_p/p\Z_p)$ is presented as
\begin{equation*}
\begin{bmatrix}
0 & 1 & 1 \\ 0 & 0 & 0
\end{bmatrix} + \begin{bmatrix}
0 & 0 \\ 0 & 1
\end{bmatrix} = \begin{bmatrix}
0 & 1 & 1 & 0 & 0 \\ 0 & 0 & 0 & 0 & 1
\end{bmatrix} = \begin{bmatrix}
0 & 0 & 0 & 1 & 1 \\ 0 & 0 & 0 & 0 & 1
\end{bmatrix}.
\end{equation*}
\end{example}

\phantom{ }

Consider an action of a permutation group $S_m$ on $\calB_m$ by $\sigma\cdot(n;H_1,\cdots,H_m):=(n;H_{\sigma(1)},\cdots,H_{\sigma(m)})$. Then $\Prs_k(\sigma\cdot(n;H_1,\cdots,H_m))$ is obtained by permuting the rows of $\Prs_k(n;H_1,\cdots,H_m)$ according to $\sigma$. For a subset $\calA\subset\calB_m$, denote $S_m\cdot\calA:=\{\sigma\cdot(n;H_1,\cdots,H_m) :\sigma\in S_m,(n;H_1,\cdots,H_m)\in\calA\} \subset \calB_m$.

Now we are ready to find a generating set of $\calC_{X_m,1,1}$. By the proof of Lemma \ref{lem3c}, for any element $(n;H_1,\cdots,H_m)\in\calC_{X_m,1,1}$, either one of the following holds:
\begin{enumerate}
    \item There exists a $1$-presentation $\calP$ of an element of $\calC_{X_m,1,1}$ such that one of the following holds.
    \begin{enumerate}
        \item$\calP+\sigma\cdot\begin{pmatrix}1^+&0&\cdots&0\end{pmatrix}^T=\Prs_1(n;H_1,\cdots,H_m)\text{ for some }\sigma\in S_m$.
        \item$\calP+\begin{pmatrix}0&\cdots&0\end{pmatrix}^T=\Prs_1(n;H_1,\cdots,H_m)$.
    \end{enumerate}
    \item There is no type $0$ on $\Prs_1(n;H_1,\cdots,H_m)$.
\end{enumerate}
We can repeat the above procedure until we find a $1$-presentation $\calP_H$ of an element of $\calC_{X_m,1,1}$ which has no type $0$ and there exist $r_1,\cdots,r_t\in\extz_{\ge0}$ and $\sigma_1,\cdots,\sigma_t\in S_m$ such that $$\calP_H+\sum_{i=1}^t\sigma_i\cdot(1;\Z_p/p^{r_i}\Z_p,1,\cdots,1)=\Prs_1(n;H_1,\cdots,H_m).$$
Since $\calP_H$ has no type $0$, it is also a $1$-presentation of an element of $\varphi_{1,0}^{-1}(\calC_{X_m,2,0})=\varphi_{1,0}^{-1}(\calB_m)$ by Proposition \ref{prop2c}(\ref{item2c4}). Conversely, by Example \ref{ex2d}, $\sigma\cdot(1;\Z_p/p^r\Z_p,1,\cdots,1)\in\calC_{X_m,1,1}$ for every $r\in\extz_{\ge0}$ and $\sigma\in S_m$. Hence we proved that $\calC_{X_m,1,1}=\langle\calA_{0,1}\cup\calA_{1,m}\rangle$ where
\begin{align*}
\calA_{r,d}&:=S_m\cdot\{(1;\Z_p/p^{r_1}\Z_p,\cdots,\Z_p/p^{r_{d}}\Z_p,\Z_p/p^r\Z_p,\cdots,\Z_p/p^r\Z_p)\in\calB_m:r_1,\cdots,r_{d}\in\extz_{\ge r}\}.
\end{align*}
for $r \in\extz_{\ge0}$ and $0 \le d\le m$. 
We can generalize this to the set $\calC_{X_m,l,1}$ for every $1 \le l \le m$. 
 
\begin{theorem}\label{thm3n}
For every $1 \le l \le m$, $\calC_{X_m,l,1}=\langle\calA_{0,l}\cup\calA_{1,m}\rangle$ and $(n;H_1,\cdots,H_m)\in\calC_{X_m,l,1}$ if and only if 
\begin{equation} \label{eq3a}
\sum_{i=1}^md_{0,i}-(m-l)\alpha_0\ge0
\end{equation}
for some (non-negative) integer $\alpha_0\ge\underset{1\le i\le m}{\max}d_{0,i}$.
\end{theorem}

\begin{proof}
Consider the sets $S_1:=\langle\calA_{0,l}\cup\calA_{1,m}\rangle$ and
\begin{equation*}
S_2:=\{(n;H_1,\cdots,H_m)\in\calB_m: \text{ the inequality } (\ref{eq3a}) \text{ holds for some integer } \alpha_0\ge\underset{1\le i\le m}{\max}d_{0,i} \}.
\end{equation*}
By Example \ref{ex2d}, we have $\calA_{0,l},\calA_{1,m}\subset\calC_{X_m,l,1}$ so $S_1 \subset\calC_{X_m,l,1}$. Lemma \ref{lem3f} implies that $\calC_{X_m,l,1}\subset S_2$. Now it is enough to show that $S_2\subset S_1$. For any element $(n;H_1,\cdots,H_m)\in S_2$, there exists a non-negative integer $\alpha_0\ge\underset{1\le i\le m}{\max}d_{0,i}$ which satisfies the inequality (\ref{eq3a}). If $d_{0,i}=0$ for each $i$, then each $H_i$ is of the form $\prod_{k=1}^{n} \Z_p/p^{r_k}\Z_p$ ($r_1, \cdots, r_n \in \extz_{\ge 1}$) so $(n; H_1, \cdots, H_m) \in \langle\calA_{1,m}\rangle$. Now assume that $d_{0,i} \ge 1$ for some $i$, so that $\alpha_0 \ge 1$. Also we may assume that $\alpha_0 = \underset{1\le i\le m}{\max}d_{0,i} \le n$.

For an integer $a$, let $[a]_{\alpha_0}$ be an integer such that $a \equiv [a]_{\alpha_0} \,\, (\bmod \; \alpha_0)$ and $1 \le [a]_{\alpha_0} \le \alpha_0$.
Choose a $1$-presentation $\Prs_1(n; H_1, \cdots, H_m)$ whose $r$-th row has type $0$ at columns $\left [a_r \right ]_{\alpha_0}$ for $\sum_{i=1}^{r-1}d_{0,i}+1 \le a_r \le \sum_{i=1}^{r}d_{0,i}$ for every $1 \le r \le m$. (See Figure \ref{fig2}.) By the inequality (\ref{eq3a}), each of the first $\alpha_0$ columns has at least $m-l$ type $0$'s so it is a $1$-presentation of an element of $\calA_{0,l}$. The remaining columns have no type $0$ so they are $1$-presentations of elements of $\calA_{1,m}$. Thus $\Prs_1(n; H_1, \cdots, H_m) = \Prs_1(n; H_1', \cdots, H_m')$ for some $(n; H_1', \cdots, H_m') \in S_1$ such that $H_i/pH_i \cong H_i'/pH_i'$ for each $i$. By the definition of the sets $\calA_{0,l}$ and $\calA_{1,m}$, if we replace a term $\Z_p/p^r\Z_p$ ($r \in \extz_{\ge 1}$) in $H_i'$ with $\Z_p/p^{r'}\Z_p$ for any $r' \in \extz_{\ge 1}$, then it is still an element of $S_1$. By iterating this process, we conclude that $(n; H_1, \cdots, H_m) \in S_1$.
\end{proof}

\begin{figure}[ht]
$$
\begin{bmatrix*}[l] 0&0&0&1^+&1^+&1^+&1^+&1^+\\0&0&1^+&0&1^+&1^+&1^+&1^+\\1^+&1^+&0&0&1^+&1^+&1^+&1^+\\ \undermat{\text{Zero}}{0\hphantom{+}&0\hphantom{+}&0\hphantom{+}&0\hphantom{+}}&\undermat{\text{One+}}{1^+&1^+&1^+&1^+} \end{bmatrix*}
$$
\caption{A choice of a $1$-presentation of $\begin{pmatrix}
8;(\Z_p/p\Z_p)^5,(\Z_p/p\Z_p)^2\times(\Z_p/p^2\Z_p)^3,\\ 
(\Z_p/p\Z_p)^2\times(\Z_p/p^3\Z_p)^4,\Z_p^4
\end{pmatrix}\in\calB_4$}
\label{fig2}
\end{figure}

Similarly, one can find a generating set of $\calC_{X_m,1,2}$ as follows. 

\begin{lemma} \label{lem3n1}
We have $\mathcal{D}_m = \mathcal{D}_{m,1} \cup \mathcal{D}_{m,2} \subset \calC_{X_m,1,2}$ for
\begin{equation*}
\begin{split}
\mathcal{D}_{m,1} & := S_m\cdot\begin{Bmatrix}
(d+1;\Z_p/p^{r_1}\Z_p,\cdots,\Z_p/p^{r_{d+2}}\Z_p,\Z_p/p\Z_p,\cdots,\Z_p/p\Z_p)\in\calB_m : \\ 0\le d\le m-2 \text{ and }r_1,\cdots,r_{d+2}\in\extz_{\ge2}  
\end{Bmatrix}, \\
 \mathcal{D}_{m,2} & :=  S_m\cdot\begin{Bmatrix}
(d+1;(\Z_p/p\Z_p)^2,\cdots,(\Z_p/p\Z_p)^2,\Z_p/p^{r_1}\Z_p,\cdots,\Z_p/p^{r_{m-d}}\Z_p)\in\calB_m : \\ 0\le d\le m \text{ and }r_1,\cdots,r_{m-d}\in\extz_{\ge2}\end{Bmatrix}.
\end{split}
\end{equation*}
\end{lemma}

\begin{proof}
For $0 \le d \le m-2$ and
$f_1(t) = \sum_{k=0}^{d+2} a_kt^k := \prod_{j=1}^{d+2} (t-x_j)$, the first integral
$$
P_1(t) := \begin{pmatrix}1&&&&pa_{d+1}t+pa_{d+2}t^2 \\
  -t&1&&&pa_rt \\
  &\ddots&\ddots&&\vdots \\ &&-t&1&pa_2t \\
  &&&-t&pa_0+pa_1t
  \end{pmatrix} \in \M_{d+1}(\Z_p)[t].
$$
satisfies $\cok(P_1(t)) \cong \cok(pf_1(t))$ so we have $\mathcal{D}_{m,1} \subset \calC_{X_m,1,2}$. For $0 \le d \le m$, the first integral
$$
P_2(t) := \begin{pmatrix}t-x_1&&&p\\ &\ddots&&\vdots\\&&t-x_d&p\\p&\cdots&p\end{pmatrix} \in \M_{d+1}(\Z_p)[t].
$$
satisfies $\cok(P_2(x_i)) \cong (\Z_p/p\Z_p)^2$ for $1 \le i \le d$ and $\cok(P_2(x_i)) \cong \Z_p/p^{r_i}\Z_p$ ($r_i \in \extz_{\ge 2}$) for $d+1 \le i \le m$ so we have $\mathcal{D}_{m,2} \subset \calC_{X_m,1,2}$.
\end{proof}

\begin{theorem} \label{thm3o}
$\calC_{X_m,1,2}=\langle\calA_{0,1}\cup\calA_{1,2}\cup\calA_{2,m}\cup\mathcal{D}_m\rangle$. Moreover, $(n;H_1,\cdots,H_m)\in\calC_{X_m,1,2}$ if and only if 
\begin{gather}
\label{eq3b} \sum_{i=1}^md_{0,i}-(m-1)\alpha_0\ge 0, \\
\label{eq3c} \left(\sum_{i=1}^md_{0,i}-(m-1)\alpha_0\right)+\left(\sum_{i=1}^m\min(d_{1,i},\alpha_1)-(m-2)\alpha_1\right)\ge 0
\end{gather}
for some non-negative integers $\alpha_0$ and $\alpha_1$ such that $\alpha_0\ge\underset{1\le i\le m}{\max}d_{0,i}$ and $2 \alpha_0+\alpha_1\ge\underset{1\le i\le m}{\max}(2d_{0,i}+d_{1,i})$.
\end{theorem}

\begin{proof}
Consider the sets $S_1:=\langle\calA_{0,1}\cup\calA_{1,2}\cup\calA_{2,m}\cup\mathcal{D}_m\rangle$ and
\begin{equation*}
S_2:=\begin{Bmatrix}(n;H_1,\cdots,H_m)\in\calB_m:
\text{ the inequalities } (\ref{eq3b}) \text{ and } (\ref{eq3c}) \text{ hold for some } \\
\alpha_0,\alpha_1 \in \Z_{\ge 0} \text{ such that } \alpha_0\ge\underset{1\le i\le m}{\max}d_{0,i}\text{ and }2 \alpha_0+\alpha_1\ge\underset{1\le i\le m}{\max}(2d_{0,i}+d_{1,i})\end{Bmatrix}.
\end{equation*}
By Example \ref{ex2d} and Lemma \ref{lem3n1}, we have $\calA_{0,1},\calA_{1,2},\calA_{2,m},\mathcal{D}_m\subset\calC_{X_m,1,2}$ so $S_1 \subset\calC_{X_m,1,2}$. (Recall that $\calC_{X_m,1,\infty} \subset \calC_{X_m,1,2}$ by Proposition \ref{prop2c}(\ref{item2c21}).) Theorem \ref{thm3g} implies that $\calC_{X_m,1,2}\subset S_2$. Now it is enough to show that $S_2\subset S_1$. For any element $(n;H_1,\cdots,H_m)\in S_2$, there exist $\alpha_0,\alpha_1 \in \Z_{\ge 0}$ such that $\alpha_0\ge\underset{1\le i\le m}{\max}d_{0,i}$, $2 \alpha_0+\alpha_1\ge\underset{1\le i\le m}{\max}(2d_{0,i}+d_{1,i})$ and the inequalities (\ref{eq3b}) and (\ref{eq3c}) hold.

Allocate $\alpha_0$ columns to \textit{Zero Zone} and $\alpha_1$ columns to \textit{One Zone}. Arrange type $0$'s on Zero Zone and type $1$'s on One Zone as in the proof of Theorem \ref{thm3n}. Then all type $0$'s are placed because $\alpha_0\ge d_{0,i}$ for each $i$, while type $1$'s may not. Fill the remaining type $1$'s to Zero Zone. If there are still remaining type $1$'s, then allocate new columns for remaining type $1$'s to \textit{Two+ Zone} and then fill type $2^+$'s on empty entries. Finally, allocate new columns to \textit{Two+ Zone} for remaining type $2^+$'s if necessary. Then we have the followings.
\begin{enumerate}[label=(\alph*)]
    \item\label{item3o1}Each column on Zero Zone contains at least $m-1$ type $0$'s.
    
    \item\label{item3o2}There are exactly $\sum_{i=1}^{m} d_{0,i}-(m-1)\alpha_0$ numbers of $\begin{pmatrix}0&\cdots&0\end{pmatrix}^T$ columns.
    
    \item\label{item3o3}By replacing at most $\sum_{i=1}^{m} d_{0,i}-(m-1)\alpha_0$ entries, we can make each column on One Zone contains at least $m-2$ type $1$'s.
    
    \item\label{item3o4}Type $1$ appears more or equal on Zero Zone than on Two+ Zone for each row.
\end{enumerate}
The properties \ref{item3o1} and \ref{item3o2} are easy to prove, and \ref{item3o3} follows from the inequality (\ref{eq3c}). For each $i$, the inequality $2\alpha_0+\alpha_1 \ge 2d_{0,i}+d_{1,i}$ is equivalent to $\alpha_0 - d_{0,i} \ge (d_{1,i}-\alpha_1)-(\alpha_0 - d_{0,i})$, which implies \ref{item3o4}.

If there are some empty entries, then swap them to the rightmost non-empty entry on each row and delete all empty columns. Then we obtain a $2$-presentation of $(n; H_1, \cdots, H_m)$, which still satisfies the properties \ref{item3o1}, \ref{item3o2}, \ref{item3o3}, and \ref{item3o4}. Figure \ref{fig3} illustrates the process to place the types.

\begin{figure}[ht]
\begin{gather*}
\begin{bmatrix} 0&0&0&0&*&*\\0&0&*&*&1&1 \\0&*&0&0&1&*\\ 
\undermat{\text{Zero}}{0&0&0&0}&\undermat{\text{One}}{*&*} \end{bmatrix}
\rightarrow \begin{bmatrix} 0&0&0&0&*&*\\0&0&1&1&1&1 \\0&*&0&0&1&*\\ 
\undermat{\text{Zero}}{0&0&0&0}&\undermat{\text{One}}{*&*} \end{bmatrix} \\ \\
\rightarrow \begin{bmatrix*}[l] 0&0&0&0&2^+&2^+&2^+&2^+\\0&0&1&1&1&1&1&1\\0&2^+&0&0&1&2^+&2^+&2^+\\\undermat{\text{Zero}}{0\hphantom{+}&0\hphantom{+}&0\hphantom{+}&0\hphantom{+}}&\undermat{\text{One}}{2^+&2^+}&\undermat{\text{Two+}}{2^+&2^+} \end{bmatrix*}
=\begin{bmatrix*}[l]0&2^+\\0&1\\0&2^+\\0&2^+\end{bmatrix*}+2\begin{bmatrix*}[l]0&2^+\\1&1\\0&2^+\\0&2^+\end{bmatrix*}+\begin{bmatrix*}[l]0\\0\\2^+\\0\end{bmatrix*}+\begin{bmatrix*}[l]2^+\\1\\1\\2^+\end{bmatrix*}
\end{gather*}
\caption{A process for $\begin{pmatrix}
8;(\Z_p/p^2\Z_p)^2\times(\Z_p/p^3\Z_p)^2,(\Z_p/p\Z_p)^6,\\ 
\Z_p/p\Z_p\times\Z_p/p^3\Z_p\times(\Z_p/p^5\Z_p)^3,\Z_p^4
\end{pmatrix}\in\calB_4$ ($\alpha_0=4$, $\alpha_1=2$)}
\label{fig3}
\end{figure}

For each column on One Zone which has $m-2-d$ type $1$ for $d>0$, concatenate $d\begin{pmatrix}0&\cdots&0\end{pmatrix}^T$; this is possible due to \ref{item3o2} and \ref{item3o3}. For each column on Two+ Zone which has type $1$ on rows $i_1,\cdots,i_d$, concatenate $e_{i_1}, \cdots, e_{i_d}$ ($e_i$ is a column on Zero Zone whose $i$-th row is $1$); this is possible due to \ref{item3o4}. These are elements of $\mathcal{D}_m$ and the other columns are elements of $\calA_{0,1}\cup\calA_{1,2}\cup\calA_{2,m}$, so we have $(n;H_1,\cdots,H_m)\in S_1$.
\end{proof}

Now we can complete the proof of Theorem \ref{thm1h} for $m=4$.

\begin{proof}[Proof of Theorem \ref{thm1h} for $m=4$]
Suppose that $(H_1, H_2, H_3, H_4)\in\calM_{\Z_p}^4$ satisfies the conditions 
$$
s := s_1 = \cdots = s_4, \; 3d_{1, i} \le D_1\text{ and }d_{1, i} + 2(d_{1,j}+d_{2,j}) \le D_1+D_2 \text{ for every } 1 \le i,j \le 4.
$$
We claim that $(s;H_1,H_2,H_3,H_4)\in\calC_{X_4,0,3}$ so that $(H_1,H_2,H_3,H_4)\in\calC_{X_4}$. By Proposition \ref{prop2c}(\ref{item2c4}), it suffices to prove that $(s;pH_1,pH_2,pH_3,pH_4)\in\calC_{X_4,1,2}$. 
Set $\alpha_0=\underset{1\le i\le 4}{\max}d_{1,i}$ and $\alpha_1=\underset{1\le i\le 4}{\max}(2d_{1,i}+d_{2,i})-2\underset{1\le i\le 4}{\max}d_{1,i}$. Then $D_1-3\alpha_0=D_1-3\underset{1\le i\le 4}{\max}d_{1,i}\ge0$ by the assumption. To apply Theorem \ref{thm3o}, we need to prove
\begin{gather*}
(D_1-3\alpha_0)+\left(\sum_{i=1}^4\min(d_{2,i},\alpha_1)-2\alpha_1\right)\ge0.
\end{gather*}
The case $F(\alpha_1):=\sum_{i=1}^4\min(d_{2,i},\alpha_1)-2\alpha_1\ge0$ is clear, so we may assume that $F(\alpha_1)<0$ and $d_{2,i} \ge \alpha_1$ for at most one $i$. 
For $i_0$ such that $2d_{1,i_0}+d_{2,i_0}=\underset{1\le i\le 4}{\max}(2d_{1,i}+d_{2,i})$, we have $\alpha_1=2d_{1,i_0}+d_{2,i_0}-2\underset{1\le i\le 4}{\max}d_{1,i}\le d_{2,i_0}$ so $d_{2,i_0}=\underset{1\le i\le 4}{\max}d_{2,i}$. 
Now we have $F(\alpha_1)=D_2-d_{2,i_0}-\alpha_1$ so
\begin{align*}
(D_1-3\alpha_0)+F(\alpha_1)&=(D_1-3\alpha_0)+(D_2-\alpha_1-d_{2,i_0})\\&=D_1+D_2-\underset{1\le i\le 4}{\max}d_{1,i}-2d_{1,i_0}-2d_{2,i_0}\ge0
\end{align*}
by the assumption. We conclude that $(s;pH_1,pH_2,pH_3,pH_4)\in\calC_{X_4,1,2}$ by Theorem \ref{thm3o}.
\end{proof}

\section{Joint distribution of multiple cokernels} \label{Sec4}

\subsection{Convergence of the joint distribution} \label{Sub41}

In this section, we study the limit
$$
\lim_{n \rightarrow \infty} \PP \begin{pmatrix}
\cok(A_n + y_i I_n) \cong H_i \\
\text{for } 1 \le i \le m
\end{pmatrix}
$$
where $A_n \in \M_n(\Z_p)$ is a Haar random matrix for each $n \ge 1$, $y_1, \cdots, y_m \in \Z_p$ are distinct and $H_1, \cdots, H_m \in \GG_p$. Although we do not know the value of the above limit, we can prove the convergence of the limit. The proof is based on the probabilistic argument in \cite[Section 2.2]{Lee23a}. 

\begin{theorem} \label{thm5a}
Let $A_n \in \M_n(\Z_p)$ be a Haar random matrix for each $n \ge 1$, $y_1, \cdots, y_m$ be distinct elements of $\Z_p$ and $H_1, \cdots, H_m \in \GG_p$. Then the limit 
$$
\lim_{n \rightarrow \infty} \PP \begin{pmatrix}
\cok(A_n + y_i I_n) \cong H_i \\
\text{for } 1 \le i \le m
\end{pmatrix}
$$
converges.
\end{theorem}

The following lemma will be frequently used in the proof of Theorem \ref{thm5a}.

\begin{lemma} \label{lem5b}
(\cite[Lemma 2.3]{Lee23a}) For any integers $n \ge r > 0$ and a Haar random $C \in \M_{n \times r}(\Z_p)$, we have
$$
\PP \left ( \text{there exists } Y \in \GL_{n}(\Z_p) \text{ such that } YC = \begin{pmatrix}
I_r \\
O
\end{pmatrix} \right ) = c_{n,r} := \prod_{j=0}^{r-1} \left ( 1 - \frac{1}{p^{n-j}} \right ).
$$
\end{lemma}

\begin{proof}[Proof of Theorem \ref{thm5a}]
For any $n \in \Z_{\ge 1}$ and $k \in \Z_{\ge 0}$, denote
$$
P_{n, k} := \PP \begin{pmatrix}
\cok(M_{A, \, [B_1, \cdots, B_k]}(y_i)) \cong H_i \\ 
\text{ for } 1 \le i \le m
\end{pmatrix} 
$$
where $A \in \M_n(\Z_p)$, $B_1, \cdots, B_k \in \M_{n \times 1}(\Z_p)$ are random and independent matrices and
$$
M_{A, \, [B_1, \cdots, B_k]}(y) := A + y \begin{pmatrix}
0 &  \\ 
 & I_{n-1}
\end{pmatrix} + \sum_{j=1}^{k} y^j \begin{pmatrix}
B_j & O_{n \times (n-1)} 
\end{pmatrix}
+ y^{k+1} \begin{pmatrix}
1 &  \\ 
 & O_{n-1}
\end{pmatrix} \in \M_n(\Z_p).
$$
To prove the convergence of the limit $\displaystyle \lim_{n \rightarrow \infty} P_{n, 0} = \lim_{n \rightarrow \infty} \PP \begin{pmatrix}
\cok(A_n + y_i I_n) \cong H_i \\
\text{for } 1 \le i \le m
\end{pmatrix}$, we will show that $P_{n, k}$ and $P_{n-1, k+1}$ are very close. For $n > 1$, $A = \begin{pmatrix}
A_1 & A_2 \\ 
A_3 & A_4
\end{pmatrix} \in \M_{1+(n-1)}(\Z_p)$, $B_1, \cdots, B_k \in \M_{n \times 1}(\Z_p)$ and 
  $$
  U = \begin{pmatrix}
    1 & O \\
    O & U_1
  \end{pmatrix} \in \GL_{n}(\Z_p) \;\; (U_1 \in \GL_{n-1}(\Z_p)),
  $$
we have
\begin{equation*}
U M_{A, \, [B_1, \cdots, B_k]}(y) U^{-1} 
= M_{A', \, [B_1', \cdots, B_k']}(y)
\end{equation*}
for
\begin{equation*}
A' = \begin{pmatrix}
    A_1 & A_2U_1^{-1} \\
    U_1A_3 & U_1A_4U_1^{-1}
  \end{pmatrix}, \, B_j'=UB_j 
\end{equation*}
by a direct computation. For a random $A_2$, the probability that there exists $U_1 \in \GL_{n-1}(\Z_p)$ such that $A_2U_1^{-1} = \begin{pmatrix}
-1 & O_{(n-2) \times 1}
\end{pmatrix}$ is $c_{n-2, 1}$ by Lemma \ref{lem5b}. Moreover, for any given $A_2$ and $U_1$, the matrices $U_1A_3$, $U_1A_4U_1^{-1}$ and $UB_j$ ($1 \le j \le k$) are random and independent if and only if $A_3$, $A_4$ and $B_j$ ($1 \le j \le k$) are random and independent. These imply that
\begin{equation} \label{eq5a}
\left | P_{n, k} - \widetilde{P}_{n, k} \right | \le 1-c_{n-2, 1}
\end{equation}
for
$$
\widetilde{\M}_{n}(\Z_p) := \left \{ \begin{pmatrix}
A_1 & -1 & O\\ 
A_2 & A_3 & A_4\\ 
A_5 & A_6 & A_7
\end{pmatrix} \in \M_{1+1+(n-2)}(\Z_p) \right \} \subset \M_n(\Z_p)
$$
and
$$
\widetilde{P}_{n, k} := \PP \begin{pmatrix}
\cok(M_{A, \, [B_1, \cdots, B_k]}(y_i)) \cong H_i \\ 
\text{ for } 1 \le i \le m
\end{pmatrix}
$$
where $A \in \widetilde{\M}_{n}(\Z_p)$, $B_1, \cdots, B_k \in \M_{n \times 1}(\Z_p)$ are random and independent. 
Let
$$
A = \begin{pmatrix}
A_1 & -1 & O\\ 
A_2 & A_3 & A_4\\ 
A_5 & A_6 & A_7
\end{pmatrix} \in \widetilde{\M}_{n}(\Z_p), \; B_j = \begin{pmatrix} c_j \\ d_j \\ e_j \end{pmatrix} \in \M_{(1+1+(n-2)) \times 1}(\Z_p) \;\; (1 \le j \le k).
$$
By elementary operations, we have
\begin{equation*}
\small
\begin{split}
M_{A, \, [B_1, \cdots, B_k]}(y) = & \, \left(\begin{array}{@{}c|c|c@{}}
A_1+\sum_{j=1}^{k} y^jc_j+y^{k+1} & -1 & O \\ \hline
A_2+\sum_{j=1}^{k} y^jd_j & A_3+y & A_4 \\ \hline
A_5+\sum_{j=1}^{k} y^je_j & A_6 & A_7+yI_{n-2}
\end{array}\right) \\
   \Rightarrow & \, \left(\begin{array}{@{}c|c|c@{}}
0 & -1 & O \\ \hline
(A_2+\sum_{j=1}^{k} y^jd_j) + (A_3+y)(A_1+\sum_{j=1}^{k} y^jc_j+y^{k+1}) & A_3+y & A_4 \\ \hline
(A_5+\sum_{j=1}^{k} y^je_j) + A_6(A_1+\sum_{j=1}^{k} y^jc_j+y^{k+1}) & A_6 & A_7+yI_{n-2} \end{array}\right) \\
   \Rightarrow & \, \left(\begin{array}{@{}c|c@{}}
A_2+A_3A_1 & A_4 \\ \hline
A_5+A_6A_1 & A_7 \end{array}\right) + y \begin{pmatrix}
0 & \\ 
 & I_{n-2}
\end{pmatrix}
+ \sum_{j=1}^{k} y^j \left(\begin{array}{@{}c|c@{}}
d_j+A_3c_j+c_{j-1} & O \\ \hline
e_j+A_6c_j & O \end{array}\right) \\
& \, + y^{k+1} \left(\begin{array}{@{}c|c@{}}
A_3+c_k & O \\ \hline
A_6 & O \end{array}\right) + y^{k+2} \begin{pmatrix}
1 & \\ 
 & O_{n-2}
\end{pmatrix} \;\; (c_0 := A_1) \\
   =: & \, M_{A', \, [B_1', \cdots, B_{k+1}']}(y).
\end{split}
\end{equation*}

\noindent Since the elementary operations do not change the cokernel, we have
$$
\cok(M_{A, \, [B_1, \cdots, B_k]}(y_i)) \cong \cok(M_{A', \, [B_1', \cdots, B_{k+1}']}(y_i))
$$
for each $i$. 
The matrices $B_1', \cdots, B_{k}'$ are given by $B_j' = \begin{pmatrix}
d_j \\ e_j
\end{pmatrix} + N_j$ ($1 \le j \le k$) where $N_1, \cdots, N_{k} \in \M_{(n-1) \times 1}(\Z_p)$ depending only on $A_1$, $A_3$, $A_6$ and $c_j$ ($1 \le j \le k$). Similarly, $\displaystyle A' = \begin{pmatrix}
A_2 & A_4 \\
A_5 & A_7 \\
\end{pmatrix} + N$ for some $N \in \M_{n-1}(\Z_p)$ depending only on $A_1$, $A_3$ and $A_6$ and $B_{k+1}' = \begin{pmatrix}
A_3+c_k & O \\
A_6 & O
\end{pmatrix}$. Therefore $A', B_1', \cdots, B_{k+1}'$ are random and independent if $d_j$, $e_j$ ($1 \le j \le k$), $A_l$ ($2 \le l \le 7$) are random and independent, or $A, B_1, \cdots, B_k$ are random and independent. This implies that
\begin{equation} \label{eq5b}
\widetilde{P}_{n, k} = P_{n-1, k+1}.
\end{equation}

Choose $M > 0$ such that $p^M H_i = 0$ for every $i$. By the equations (\ref{eq5a}) and (\ref{eq5b}), we have
\begin{equation} \label{eq5c}
\begin{split}
\left | P_{n,0} - P_{n-d, d} \right |
\le \sum_{i=1}^{d} \left | P_{n-i+1, i-1} - P_{n-i, i} \right | 
\le \sum_{i=1}^{d} (1 - c_{n-i-1, 1}) 
< \frac{2}{p^{n-d-1}}
\end{split}
\end{equation}
for $d \ge M$ and $n>d$. For every $n > M+1$, we have $P_{n-M, M} = P_{n-M, M+1}$ and
\begin{equation*}
\left | P_{n,0} - P_{n+1, 0} \right | 
\le \left | P_{n,0} - P_{n-M, M} \right | + \left | P_{n+1,0} - P_{n-M, M+1} \right | 
< \frac{4}{p^{n-M-1}}
\end{equation*}
by the equation (\ref{eq5c}). This finishes the proof.
\end{proof}

\subsection{Mixed moments} \label{Sub42}

Now we compute the mixed moments of the cokernels $\cok(A_n + px_iI_n)$ ($1 \le i \le m$) where each random matrix $A_n \in \M_n(\Z_p)$ ($n \ge 1$) is given as in Theorem \ref{thm1e}. 
Nguyen and Van Peski \cite{NVP24} and the second author \cite{Lee24} independently defined mixed moments of multiple random groups and extended the universality results of Wood \cite[Theorem 1.3]{Woo19} to the joint distribution of various multiple cokernels. 
The \textit{mixed moments} of (not necessarily independent) random finite groups $Y_1, \cdots, Y_r$ are defined to be the expected values
$$
\EE \left ( \prod_{k=1}^{r} \# \Sur(Y_k, G_k) \right ) 
$$
for finite groups $G_1, \cdots, G_r$.

For a partition $\lambda = (\lambda_1 \ge \cdots \ge \lambda_r)$, let $\lambda'$ be the conjugate of $\lambda$, $G_{\lambda} := \prod_{i=1}^{r} \Z / p^{\lambda_i} \Z$ be the finite abelian $p$-group of type $\lambda$ and denote $m(G_{\lambda}) :=  p^{\sum_{i} \frac{(\lambda'_i)^2}{2}}$. The following theorem is a special case of \cite[Theorem 1.3]{Lee24} (taking $P=\left \{ p \right \}$), which extends \cite[Theorem 2.5]{Woo22} to the multiple random groups. We note that Nguyen and Van Peski \cite[Theorem 9.1]{NVP24} independently obtained a similar result.

\begin{theorem} \label{thm6a}
(\cite[Theorem 1.3]{Lee24}) Let $Y = (Y^{(1)}, \cdots, Y^{(r)})$, $Y_n = (Y_{n}^{(1)}, \cdots, Y_{n}^{(r)})$ ($n \ge 1$) be random $r$-tuples of elements in $\GG_p$. Suppose that for every $G^{(1)}, \cdots, G^{(r)} \in \GG_p$, we have
\begin{equation*}
\lim_{n \rightarrow \infty} \EE(\prod_{k=1}^{r} \# \Sur(Y_{n}^{(k)}, G^{(k)}))
= \EE(\prod_{k=1}^{r} \# \Sur(Y^{(k)}, G^{(k)}))
= O \left ( \prod_{k=1}^{r} m(G^{(k)}) \right ).
\end{equation*}
Then for every $H^{(1)}, \cdots, H^{(r)} \in \GG_p$, we have
$$
\lim_{n \rightarrow \infty} \PP (Y_{n}^{(k)} \cong H^{(k)} \text{ for } 1 \le k \le r)
= \PP (Y^{(k)} \cong H^{(k)} \text{ for } 1 \le k \le r).
$$    
\end{theorem}

\begin{example} \label{ex6b}
(\cite[Section 2.2]{Woo22}) Let $Y_{\text{odd}}$ and $Y_{\text{even}}$ be random elements of $\GG_p$ given as in \cite[(2.7)]{Woo22}. (We consider them as random finite abelian $p$-groups which are always elementary abelian $p$-groups.) Let $Y_1^{(1)}, \cdots, Y_1^{(r)}$ (resp. $Y_2^{(1)}, \cdots, Y_2^{(r)}$) be i.i.d. random variables in $\GG_p$ following the distribution of $Y_{\text{odd}}$ (resp. $Y_{\text{even}}$). Then we have
$$
\EE(\prod_{k=1}^{r} \# \Sur(Y_1^{(k)}, (\Z/p\Z)^t))
= \EE(\prod_{k=1}^{r} \# \Sur(Y_2^{(k)}, (\Z/p\Z)^t))
= p^{\frac{r(t^2+t)}{2}}
$$
by \cite[Theorem 2.8]{Woo22}. 
This example shows that Theorem \ref{thm6a} can fail even if the mixed moments are slightly larger than the upper bound, which is given by $O(\prod_{k=1}^{r} m((\Z/p\Z)^t)) = O(p^{\frac{rt^2}{2}})$ here.
\end{example}

Let $P_1, \cdots, P_m \in \Z_p[t]$ be monic polynomials whose reductions modulo $p$ are irreducible and $A_n \in \M_n(\Z_p)$ be a random matrix for each $n \ge 1$. 
Assume that one can determine the (limiting) joint distribution of the cokernels $\cok(P_i(A_n))$ ($1 \le i \le m$) when each $A_n$ is equidistributed. Then the next goal would be to prove universality of the joint distribution of the cokernels for general $\varepsilon_n$-balanced matrices $A_n$. 
The only known way to prove such universality is to compute the mixed moments of the cokernels. Recall that $X_m = \left \{ x_1, \cdots, x_m \right \}$ is a finite ordered subset of $\Z_p$ whose elements have distinct reductions modulo $p$.

\begin{theorem} \label{thm6c}
Let $(\varepsilon_n)_{n \ge 1}$ be a sequence of real numbers such that for every $\Delta > 0$, we have $\varepsilon_n \ge \frac{\Delta \log n}{n}$ for sufficiently large $n$. Let $A_n \in \M_n(\Z_p)$ be an $\varepsilon_n$-balanced random matrix for each $n \ge 1$, $G_1, \cdots, G_m \in \GG_p$ and $p_k : \prod_{i=1}^{m} G_i \rightarrow G_k$ ($1 \le k \le m$) be the natural projections. Then we have
\begin{equation} \label{eq6a}
\lim_{n \rightarrow \infty} \EE \left ( \prod_{i=1}^{m} \# \Sur(\cok(A_n+px_iI_n), G_i) \right ) = \left | S_{G_1, \cdots, G_m}(X_m) \right |
\end{equation}
where $T_x \in \End(\prod_{i=1}^{m} G_i)$ ($(g_1, \cdots, g_m) \mapsto (x_1g_1, \cdots, x_mg_m)$) and 
$$
S_{G_1, \cdots, G_m}(X_m) := \left \{ G \le \prod_{i=1}^{m} G_i : p_i(G)=G_i \text{ for each } i \text{ and } p T_{x}(G) \le G \right \}.
$$
\end{theorem}

\begin{proof}
Choose $k \in \Z_{\ge 1}$ such that $p^kG_i = 0$ for all $i$. Let $R = \Z/p^k\Z$, $A_n' \in \M_{n}(R)$ be the reduction of $A_n$ modulo $p^k$ (which is also $\varepsilon_n$-balanced) and $v_j = A_n' e_j \in R^n$ where $\left\{ e_1, \cdots, e_{n} \right\}$ is the standard basis of $R^n$. Then we have
\begin{equation} \label{eq6b}
\begin{split}
& \EE \left ( \prod_{i=1}^{m} \# \Sur(\cok(A_n+px_iI_n), G_i) \right ) \\
= \, & \sum_{\substack{F_i \in \Sur(R^n, G_i) \\ 1 \le i \le m}} \PP(F_i(v_j+px_ie_j)=0 \text{ for all } 1 \le j \le n) \\
= \, & \sum_{\substack{F_i \in \Sur(R^n, G_i) \\ 1 \le i \le m}} \PP(Fv_j = -pT_{x}(Fe_j) \text{ for all } 1 \le j \le n ) \\
= \, & \sum_{\substack{F_i \in \Sur(R^n, G_i) \\ 1 \le i \le m}} \PP(FA_n' = -pT_{x}F).
\end{split}
\end{equation}
If the probability $\PP(FA_n' = -pT_{x}F)$ is non-zero, then $G = \im(F)$ is an element of $S_{G_1, \cdots, G_m}(X_m)$. Following the proof of \cite[Theorem 4.12]{NW22a}, one can prove that there are constants $c, K>0$ (depend only on $G$ and $X_m$) such that
\begin{equation} \label{eq6c}
\left| \sum_{F \in \Sur_R(R^n, G)} \PP(FA_n' = -pT_{x}F) - 1 \right| \le K n^{-c}
\end{equation}
for every $n \ge 1$ and $G \in S_{G_1, \cdots, G_m}(X_m)$. (To do this, we need to generalize \cite[Lemma 4.11]{NW22a} to an upper bound of $\PP(FX=A)$ for every $A \in \im (F)$. For any $X_0$ such that $FX_0=A$, we have $\PP(FX=A)=\PP(F(X-X_0)=0)$ and $X-X_0$ is also an $\varepsilon_n$-balanced matrix so this immediately follows from the case $A=0$.)

Now the equations (\ref{eq6b}) and (\ref{eq6c}) imply that
\begin{equation*}
\begin{split}
\lim_{n \rightarrow \infty} \EE \left ( \prod_{i=1}^{m} \# \Sur(\cok(A_n+px_iI_n), G_i) \right )
& = \lim_{n \rightarrow \infty} \sum_{\substack{F_i \in \Sur(R^n, G_i) \\ 1 \le i \le m}} \PP(FA_n' = -pT_{x}F) \\
& = \lim_{n \rightarrow \infty} \sum_{G \in S_{G_1, \cdots, G_m}(X_m)} \sum_{F \in \Sur(R^n, G)} \PP(FA_n' = -pT_{x}F) \\
& = \left | S_{G_1, \cdots, G_m}(X_m) \right |. \qedhere
\end{split}   
\end{equation*}
\end{proof}

\begin{example} \label{ex6d}
Let $p \ge m \ge 3$, $G_1=\cdots =G_m=(\Z/p\Z)^t$ and $\left \{ e_1, \cdots, e_t \right \}$ be the standard basis of $(\Z/p\Z)^t$. Then we have

\begin{equation*}
\begin{split}
\left | S_{G_1, \cdots, G_m}(X_m) \right | & := \# \left \{ G \le \prod_{i=1}^{m}G_i : p_i(G)=G_i \text{ for each } i \right \} \\
& \ge \# \begin{Bmatrix}
G = \left \langle (e_j, u_{2,j}, \cdots, u_{m,j} ) : 1 \le j \le t \right \rangle : \\
\left \langle u_{i, 1}, \cdots, u_{i, t} \right \rangle
= (\Z/p\Z)^t \text{ for every } 2 \le i \le m
\end{Bmatrix} \\
& = \left ( \prod_{k=0}^{t-1} (p^t-p^k) \right )^{m-1} \\
& > c_{\infty}(p)p^{(m-1)t^2}.
\end{split}    
\end{equation*}
To apply Theorem \ref{thm6a}, the mixed moments of the cokernels for $G_1, \cdots, G_m$ should be 
$$
O \left ( \prod_{i=1}^{m} m(G_i) \right ) = O((p^{\frac{t^2}{2}})^m) = O(p^{\frac{mt^2}{2}}).
$$
However, the above inequality implies that for every constant $C>0$, we have 
$$
\left | S_{G_1, \cdots, G_m}(X_m) \right | > c_{\infty}(p)p^{(m-1)t^2} > Cp^{\frac{mt^2}{2}} 
$$
for sufficiently large $t$. Therefore we cannot apply Theorem \ref{thm6a} in this case. In fact, Example \ref{ex6b} tells us that there are two different $m$-tuples of random elements in $\GG_p$ whose mixed moments for $G_1= \cdots = G_m = (\Z/p\Z)^t$ are $p^{\frac{m(t^2+t)}{2}}$, which is smaller than $c_{\infty}(p)p^{(m-1)t^2}$ for every $t \ge 4$ by the inequality $c_{\infty}(p) > \frac{1}{4}$.
\end{example}

By the above example, we cannot determine the joint distribution of the cokernels $\cok(A_n+px_iI_n)$ ($1 \le i \le m$) for $m \ge 3$ using existing methods. 
As we mentioned in the introduction, we believe that one needs to combine combinatorial relations between the cokernels (Theorem \ref{thm1h} and Conjecture \ref{conj1i}) and the mixed moments of the cokernels (Theorem \ref{thm6c}) to solve this problem.

\section*{Acknowledgments}

J. Jung was partially supported by Samsung Science and Technology Foundation (SSTF-BA2001-04). J. Lee was supported by the new faculty research fund of Ajou University (S-2023-G0001-00236). The authors thank Gilyoung Cheong and Seongsu Jeon for their helpful comments.

\end{document}